\documentclass[a4paper, 11pt, oneside,reqno]{article}

\usepackage[utf8]{inputenc}
\usepackage[english]{babel}
\usepackage{caption}
\usepackage{float}
\usepackage{graphicx}
\usepackage{array}
\usepackage{bm} 
\usepackage{hyperref} 

\usepackage[left=1in, right=1in, top=1in, bottom=1in, includefoot, headheight=13.6pt]{geometry}
\usepackage{url}
\usepackage{datetime}
\usepackage{times}
\usepackage{xcolor}
\usepackage{amsmath,amssymb}
\usepackage{booktabs} 
\usepackage[titletoc,title]{appendix} 
\usepackage{geometry} 
\usepackage{ulem} 

\usepackage{setspace}
\usepackage{rotating} 
\usepackage{enumitem} 
\usepackage{changepage} 

\usepackage{cases} 

\usepackage[amsmath,thmmarks]{ntheorem}

\newtheorem{theorem}{Theorem}
\newtheorem{example}[theorem]{Example}
\newtheorem{remark}[theorem]{Remark}
\newtheorem{corollary}[theorem]{Corollary}
\newtheorem{lemma}[theorem]{Lemma}
\newtheorem{definition}[theorem]{Definition}

\newtheorem{proposition}[theorem]{Proposition}

\theoremstyle{nonumberplain}
\theorembodyfont{\normalfont}
\theoremsymbol{\ensuremath{\square}}
\newtheorem{proof}{Proof}

\makeatletter
\newcommand{\leqnomode}{\tagsleft@true}
\newcommand{\reqnomode}{\tagsleft@false}
\makeatother

\newdateformat{dateenglish}{\monthname[\THEMONTH]~\THEDAY,~\THEYEAR}
\DeclareMathOperator{\sgn}{sgn}

\DeclareMathOperator{\Var}{Var} 
\DeclareMathOperator{\Cov}{Cov}

\let\P\relax 
\DeclareMathOperator{\P}{\mathbb{P}} 
\DeclareMathOperator{\E}{\mathbb{E}}
\def\N{\mathbb{N}}
\def\R{\mathbb{R}}
\def\C{\mathbb{C}}
\def\Z{\mathbb{Z}}
\newcommand{\1}{\mbox{1\hspace{-0.28em}I}}
\newcommand\numberthis{\addtocounter{equation}{1}\tag{\theequation}}

\title{Joint FCLT for Sample Quantile and Measures of Dispersion \\for Functionals of Mixing Processes}

\author{\Large Marcel Br\"autigam$^{\,\ast}$ and Marie Kratz$^{\dagger}$\\[1ex]
\small $~^\ast$ CREAR, ESSEC Business School\\ 
\small $^{\dagger}$ ESSEC Business School, IDO dept \& CREAR, and Labex MME-DII, CY initiative}

\date{}

\begin{document}
\maketitle

\begin{abstract}
\noindent In this paper, we establish a joint (bivariate) functional central limit theorem of the sample quantile and the $r$-th absolute centred sample moment for functionals of mixing processes. 
More precisely, we consider $L_2$-near epoch dependent processes that are functionals of either $\phi$-mixing or absolutely regular processes. 
The general results we obtain can be used for two classes of popular and important processes in applications: The class of augmented GARCH($p$,$q$) processes with independent and identically distributed innovations (including many GARCH variations used in practice) and the class of ARMA($p$,$q$) processes with mixing innovations (including, e.g., ARMA-GARCH processes).  
For selected examples, we provide exact conditions on the moments and parameters of the process for the joint asymptotics to hold.\\[5ex]

\end{abstract}
 
\section{Introduction and framework}
\label{sec:Intro}

The focus of this paper is on the asymptotic theory of sample estimators of standard statistical quantities as rank, location and dispersion measures, for a very large class of widely used stationary processes. Such a theory is often needed for related statistical inference. The literature for sample quantiles has considerably developed (closely related to empirical processes) from the iid setting toward dependent random variables. In view of applications, it is also important to consider jointly the estimators of such standard statistical quantities. Fields of applications include dynamic systems (networks), (financial) econometrics, reliability, risk theory and management, quantitative finance, etc.
For instance, dispersion measures, as e.g. the second centred moment (named volatility in financial fields), are essential in economics for studying market fluctuations. This is why they are often introduced, via conditioning or jointly, in the econometric analysis of macroeconomic and financial data, to take into account the market state. 
Investigations on the relation between volatility estimators (a measure of dispersion) and Value-at-Risk (a quantile estimator) for time-series models are common in this field; see e.g. 
\cite{Brautigam22}, 
\cite{Giot04}, \cite{Jeon13}, \cite{Zumbach12}.

In the literature, the asymptotic behavior of each quantity has been investigated on its own, under some dependence structures. 
Limit theorems for sample quantiles have been proposed, inter alia, in the case of m-dependent or $\phi$-mixing processes (\cite{Sen68}, \cite{Sen72}), infinite-variance or non-linear processes (\cite{Hesse90}, \cite{Wu05}), and functionals of mixing processes (\cite{Wendler11}).
They are generally based on the use of the Bahadur representation developed in \cite{Bahadur66} for independent and identically distributed (iid) random variables (and refined by Kiefer in a sequence of papers), then extended (generally via coupling techniques) to linear and nonlinear processes, with various types of dependence (see e.g. \cite{Wu05,Wu21} and references therein).

For dispersion measures, there exist results for the mean absolute deviation (MAD) when considering strongly mixing processes (\cite{Segers14}), while the asymptotics of the sample variance can be found for various settings as examples in standard textbooks, see e.g. \cite{Anderson71}, 
\cite{Brockwell91}, \cite{VanDerVaart98}.
When considering more generally the autocovariance function of such estimators, we can mention, exemplarily, results for linear, bilinear or non-linear  processes with regularly varying noise (\cite{Davis98}, \cite{Davis99}), and for long-memory time series (e.g.~\cite{Horvath08} - and for a more complete overview~\cite{Kulik12} and references therein). 

In this paper, we establish the joint asymptotics of the sample quantile and  
the $r$-th absolute centred sample moment for functionals of strictly stationary processes $(Z_n)_{n \in \Z}$, namely
$$
X_n = f (Z_{n+k}, k \in \Z), \quad \text{with $f$ measurable,}
$$ 
assuming $(Z_n)$ $\phi$-mixing or absolutely regular, respectively, and $(X_n)$ near-epoch dependent. 
We briefly recall these dependence notions further below.
Our limit theorem extends existing results on joint asymptotics in the case of an identically and independently distributed (iid) sample (see~\cite{Brautigam22}). In particular, we establish an asymptotic representation for the $r$-th absolute centred sample moment for stationary, ergodic, short-memory processes under minimal conditions (extending the results in~\cite{Segers14} and \cite{Brautigam22}). 
Furthermore, our general results include two important popular classes of linear and nonlinear processes: The class of augmented GARCH($p$,$q$) processes with iid innovations and ARMA($p$,$q$) processes with mixing innovations.
The former, introduced by Duan in \cite{Duan97}, contains many well-known GARCH processes as special cases, of standard use, e.g. in quantitative finance and financial econometrics. Previous works on univariate CLTs and stationarity conditions for this class of GARCH processes are, among others, \cite{Aue06}, \cite{Berkes08}, \cite{Hormann08} and \cite{Lee14}. 
The class of ARMA processes includes the classical ARMA models with iid or white noise innovations (see e.g. \cite{Brockwell91}, \cite{Box15} for classical references on the topic), but also with mixing inovations. The earliest explicitly mentioned examples of ARMA processes with mixing innovations date back,  to our knowledge, to \cite{Weiss84}, \cite{Weiss86} considering ARMA models with ARCH errors. 
More widespreadly used in applications is the ARMA-GARCH model, which can be seen as an extension of the GARCH models to also have dynamics in the mean and not only in the variance process. First contributions on such types of processes can be found in \cite{Ling97}, \cite{Ling98}, \cite{Ling03}; more recent examples of applications are, e.g. \cite{Hoga19,Song18,Spierdijk16}. 
Analogously to the class of augmented GARCH processes, there exist results in the literature on FCLTs for ARMA processes with mixing innovations (or even larger classes of linear processes), as e.g. \cite{Davidson02}, \cite{Qiu11}.

Classical and broad notions for (asymptotically vanishing) dependence of a process are known under the name of mixing (see \cite{Bradley05} for an overview). 
Let us recall some notions of weak dependence related to the mixing coefficients. 
Let $(Y_n)_{n \in \Z}$ be a random process and denote the corresponding $\sigma$-algebra as $\mathcal{F}_{Y,s,t} = \sigma(Y_u; s \leq u  \leq t)$.
\begin{definition}[Mixing coefficients] \label{def:mixing_notions}
For any integer $n\geq 1$, let
\vspace{-2ex}
\begin{align*} 
\alpha(n) &:= \sup_{j \in \Z}\; \sup_{C \in \mathcal{F}_{Y,-\infty, j}, D \in \mathcal{F}_{Y,j+n, \infty}} \lvert P(C \cap D) - P(C)P(D) \rvert, 
\\ \beta(n) &:= \sup_{j \in \Z} \;\E[ \sup_{C \in \mathcal{F}_{Y,j+n, \infty} } \lvert P(C \vert \mathcal{F}_{Y,-\infty,j}) - P(C) \rvert], 
\\ \phi(n) &:= \sup_{j \in \Z} \;\sup_{C \in \mathcal{F}_{Y,-\infty,j} : P(C)>0, D \in \mathcal{F}_{Y,j+n, \infty}} \lvert P(D \vert C) - P(D) \rvert. 
\end{align*}
The process $(Y_n)_{n \in \Z}$ is called strongly mixing 
if $\displaystyle \alpha(n) \underset{n\to\infty}{\rightarrow}0$, $\displaystyle \beta$-mixing or absolutely regular
if $\beta(n) \underset{n\to\infty}{\rightarrow} 0$, and $\phi$-mixing 
if $\phi(n) \underset{n\to\infty}{\rightarrow} 0$.
\end{definition}
qNote that, for a stationary time series, we can omit the $\displaystyle \sup_{j \in \Z}$ in the definition of the mixing notions and simply set, w.l.o.g. , $j=0$. Also, $\phi$-mixing implies $\beta$-mixing that implies $\alpha$-mixing, but the converse does not hold in general.

A drawback of mixing is that, in general, a functional that depends on an infinite number of lags of a mixing process will not be mixing itself. This gave rise to the introduction of the notion of $L_p$-near-epoch dependence ($L_p$-NED) (that goes back to~\cite{Ibragimov62} in the case of $p=2$ for an underlying $\phi$-mixing process, even if not named NED there). It imposes additional structure on functionals of mixing processes enabling statistical inference for this larger class of processes. 
Alternatively, $L_p$-NED is also called $p$-approximation functional; see \cite{Borovkova01}, \cite{Wendler11}.
\begin{definition}[$L_p$-NED, \cite{Andrews88}]
For $p>0$, a stationary process $(X_n)_{n \in \Z}$ is called $L_p$-NED on a process $\left( Z_n\right)_{ n\in \Z }$ if, for $k \geq 0$,
\[ \| X_0 - \E[X_0 \vert \mathcal{F}_{Z,-k, +k}] \|_p \leq \nu(k), \]
for non-negative constants $\nu(k)$ such that $\nu(k) \rightarrow 0$ as $k \rightarrow \infty$.
\end{definition}
Note that, in the econometrics literature, a specific terminology is used for the rate of $\nu(k)$: If $\nu(k)= O(k^{-\tau -\epsilon})$ for some $\epsilon >0$, one calls the process $L_p$-NED of size $\left(-\tau\right)$ and, if $\nu(k)=O(e^{-\delta k})$ for some $\delta>0$, geometrically $L_p$-NED. 

Finally, let us present some technical conditions on 
the marginal distribution function $F_X$ of the stationary process $(X_n)_{n \in \Z}$ under study, and give the notation of the sample estimators for the two quantities of interest. 
We denote, whenever they exist, the probability density function (pdf) of $F_X$ by $f_X$, with mean $\mu$, variance $\sigma^2$, as well as, for any integer $r\geq 1$, the $r$-th absolute centred moment, $\mu(X,r) := \E[\lvert X_0 - \mu \rvert^r]$. The quantile of order $p$ of $F_X$ is defined as $q_X(p):= \inf \{ x \in \R: F_X(x) \geq p \}$.

We impose four different types of conditions on the marginal distribution $F_X$, as in the iid case (see \cite{Brautigam22}).
First, we assume the existence of its finite $2k$-th moment for any integer $k>0$.
Then, the continuity or $l$-fold differentiability of $F_{X}$ (at a given point or neighbourhood) for any integer $l> 0$, 
and the positivity of $f_X$ (at a given point or neighbourhood). Those conditions are named as:
\begin{align*}
&(M_k) &&\E[\lvert X_0 \rvert^{2k}] < \infty, 
\\ &(C_0) && F_X \text{~is continuous}, 
\\ \phantom{text to make the distance less} &(C_l^{~'}) &&F_X \text{~is~} l\text{-times differentiable,} \phantom{text to make the distance between \& and \& \& less} 
\\ &(C_1^{+}) &&f_X \text{~is positive.} 
\end{align*}
Given a sample $X_1,\cdots,X_n$, with order statistics $X_{(1)}\le \cdots\le X_{(n)}$, let $q_n (p)$ denote the sample quantile of order $p \in [0,1]$, namely 
$$ q_n (p) = X_{( \lceil np \rceil )},$$ 
where $\lceil x \rceil =   \min{ \{ m \in \mathbb{Z}  : m \geq x \} }$ and $[x]$ are the rounded-up integer-part and the nearest-integer of a real number $x \in \R$, respectively. 

The $r$-th absolute centred sample moment is denoted by $\hat{m}(X,n,r)$ and defined, for $r \in \N$, by
\begin{equation}\label{eq:sampleMDisp}
\hat{m}(X,n,r) := \frac{1}{n} \sum_{i=1}^n \lvert X_i -  \bar{X}_n \rvert^r,
\end{equation}
$\bar{X}_n$ representing the empirical mean.
Special cases of this latter estimator include the sample variance ($r=2$) 
and the sample mean absolute deviation around the sample mean ($r=1$).

Recall the standard notation $u^T$ for the transpose of a vector $u$ and $\displaystyle \sgn(x) := -\1_{(x<0)}+\1_{(x>0)}$ for the signum function.
By $\lvert \cdot \rvert$, we denote the euclidean norm, and the usual $L_p$-norm is denoted by $\| \cdot \|_p := \E^{1/p}[ \lvert \cdot \rvert^p]$.
Moreover the notations $\overset{d}\rightarrow$, $\overset{a.s.}\rightarrow$, $\overset{P}\rightarrow$ and $\overset{D_d[0,1]}\rightarrow$ correspond to the convergence in distribution, almost surely, in probability and in distribution of a random vector in the d-dimensional Skorohod space $D_d[0,1]$. Further, for real-valued functions $f$ and $g$, we write $f(x) = O(g(x))$ (as $x \rightarrow \infty)$ if and only if there exists a positive constant $M$ and a real number $x_0$ s.t. $\lvert f(x) \rvert \leq M g(x)$ for all $x \geq x_0$, and $f(x)=o(g(x))$ (as $x \rightarrow \infty$) if, for all $\epsilon>0$, there exists a real number $x_0$ s.t., for all $x \geq x_0$, $\lvert f(x) \rvert \leq \epsilon g(x)$. Analogously, for a sequence of rv's $X_n$ and constants $a_n$, we denote by $X_n = o_P(a_n)$ the convergence in probability to 0 of $X_n/a_n$.

The structure of the paper is as follows. We present in Section~\ref{sec:asympt_results} the main results on the bivariate FCLT for the sample quantile and the $r$-th absolute centred sample moment for functionals of $\phi$-mixing or absolutely regular processes. 
In Section~\ref{sec:exp}, we apply our general results to the family of augmented GARCH($p$, $q$) processes and ARMA($p$,$q$) processes with mixing innovations. 
Section~\ref{sec:aux_res} gathers auxiliary results needed for the proofs of the main results, but also of interest on their own: the asymptotic representation of the $r$-th absolute centred sample moment and results on $L_2$-NED of bounded and unbounded functionals. 
Finally, the proofs are given in Section~\ref{sec:proofs}. 

\section{The Bivariate FCLT}
\label{sec:asympt_results}

Let us state the main result. To ease its presentation, we introduce a trivariate normal random vector $(U, V, W)^T$, functional of a random process $X$, with mean zero and the following covariance matrix:
{\small
\begin{equation*}
    (Co)~\left\{\begin{aligned}
\Var(U) &= \Var(X_0) +2 \sum_{i=1}^{\infty} \Cov(X_i, X_0)
\\ \Var(V) &= \Var(\lvert X_0 \rvert^r) +2 \sum_{i=1}^{\infty}  \Cov(\lvert X_i\rvert^r, \lvert X_0 \rvert^r) 
\\ \Var(W) &= \Var \left( \frac{p- \1_{(X_0 \leq q_X(p))}}{f_X(q_X(p))} \right) + 2 \sum_{i=1}^{\infty} \Cov \left( \frac{p- \1_{(X_i \leq q_X(p))}}{f_X(q_X(p))}, \frac{p- \1_{(X_0 \leq q_X(p))}}{f_X(q_X(p))} \right)
\\ &= \frac{p(1-p)}{f_X^2(q_X(p))} + \frac{2}{f_X^2(q_X(p))} \sum_{i=1}^{\infty} \left( \E[\1_{(X_0 \leq q_X(p))} \1_{(X_i \leq q_X(p))} ] -p^2  \right) \notag
\\ \Cov(U,V) &= \sum_{i \in \mathbb{Z}} \Cov(\lvert X_i \rvert^r,X_0) = \sum_{i \in \Z} \Cov(\lvert X_0  \rvert^r, X_i  )
\\ \Cov(U,W) &= \frac{-1}{f_X(q_X(p))} \sum_{i \in \mathbb{Z}} \Cov( \1_{(X_i \leq q_X(p))},X_0) = \frac{-\sum_{i \in \Z} \Cov( \1_{(X_0 \leq q_X(p))},X_i )}{f_X(q_X(p))}
\\ 
\Cov(V,W) &= \frac{-1}{f_X(q_X(p))} \sum_{i \in \mathbb{Z}} \Cov(  \lvert X_0 \rvert^r, \1_{(X_i \leq q_X(p))}) =\frac{-\sum_{i \in \Z} \Cov(\lvert X_i \rvert^r, \1_{(X_0 \leq q_X(p))} )}{f_X(q_X(p))}
 \end{aligned}
    \right.
\end{equation*}
}

\begin{theorem}[bivariate FCLT] \label{thm:biv_FCLT}
Consider a process $(X_n)_{n \in \Z}$ that can be represented as a  functional $X_n=f (Z_{n+k}, k \in \Z)$ of a strictly stationary process $(Z_n)_{n \in \Z}$ and introduce the random vector 
$$ 
T_{n,r}(X) = \begin{pmatrix} q_n(p) -q_X(p) \\ \hat{m}(X,n,r) -m(X,r) \end{pmatrix},\; \text{with integer } r>0. 
$$
Assume for the marginals that the conditions $(M_r)$, $(C_1^{+})$ at $q_X(p)$ and $(C_2^{~'})$ in a neighbourhood of $q_X(p)$ and, for $r=1$, $(C_0)$ at $\mu$, are satisfied. 

In terms of dependence, suppose that: \\[-5ex]
\begin{itemize}
\item[$(D1)$] $(X_n)_{n \in \Z}$ is $L_2$-NED with constants $\nu(k) =  O(k^{-s}), s>6$, on a $\phi$-mixing process $(Z_n)_{n \in \Z}$ with mixing coefficient $\beta(n)=O(n^{-\kappa})$, for some $\kappa >3$.
\item[(D2)] For $r >1$, $\lvert X_n^{r} \rvert$ is $L_2$-NED with constants $\nu(k) =  O(k^{-\gamma})$, for $\gamma >2$. 
\end{itemize}
Then, we have, for $t \in [0,1]$, 
\[  \sqrt{n}~t ~ T_{[nt],r}(X) \overset{D_2[0,1]}{\underset{n \rightarrow \infty}{\rightarrow}}  \textbf{W}_{\Gamma^{(r)}} (t), \]
where $(\textbf{W}_{\Gamma^{(r)}}(t))_{t \in [0,1]}$ is the 2-dimensional Brownian motion with covariance matrix $\Gamma^{(r)} \in \R^{2\times 2}$ defined, for any $(s,t) \in [0,1]^2$, by $\Cov(\textbf{W}_{\Gamma^{(r)}}(t),\textbf{W}_{\Gamma^{(r)}}(s)) = \min(s,t) \Gamma^{(r)}$, where
\begin{align*}
\Gamma_{11}^{(r)}&= \Var(W),
\\ \Gamma_{22}^{(r)} &= r^2 \E[ X_0^{r-1} \sgn(X_0)^r]^2 \Var(U) + \Var(V) - 2r \E[ X_0^{r-1} \sgn(X_0)^r] \Cov(U,V),
\\  \Gamma_{12}^{(r)}&= \Gamma_{21}^{(r)} = -r\E[ X_0^{r-1} \sgn(X_0)^r] \Cov(U,W) + \Cov(V,W),
\end{align*}
$(U, V, W)^T$ being the trivariate normal vector (functional of $X$) with mean zero and covariance given in $(Co)$, all series being absolute convergent.
\end{theorem}

\begin{remark}\label{rmk:rmk_after_thm}~
\begin{enumerate}
\item The conditions of Theorem~\ref{thm:biv_FCLT} are those required to establish a univariate CLT for each component of $T_{[nt],r}(X) $, namely the sample quantile and the $r$-th absolute centred sample moment. For the latter statistic, it requires first establishing a suitable representation, a new result on its own, presented in Proposition~\ref{prop:Abs_central_mom_asympt_garch}. Requiring $(C_2^{~'})$ and $(C_1^{+})$ in a neighbourhood of $q_X(p)$ corresponds to the conditions for the CLT of the sample quantile of a stationary process, which is $L_1$-NED with constants $\nu(k) =O(k^{-(\beta+3)})$ on an absolutely regular process (thus also for a $\phi$-mixing process) with mixing rate $O(n^{-\beta})$, for $\beta >3$ - see Corollary~1 in \cite{Wendler11}. \\[1ex]
Further, the $L_2$-NED with constants $\nu(k) =O(k^{-(\beta+3)})$ on a $\phi$-mixing process with mixing rate $O(n^{-\beta})$, for $\beta >3$, together with $(M_r)$ and the additional assumptions in the cases of $r=1$ ($(C_0)$ at $\mu$) and $r>1$ ($L_2$-NED with constants $\nu(k) =O(k^{-\gamma}), \gamma >2$, on $(\lvert X_n \rvert^r)_{n \in \Z}$), are sufficient conditions for the univariate CLT of the $r$-th centred sample moment (using Theorem 1.2 in \cite{Davidson02}, which is a special case of Theorem~3.1 in \cite{DeJong00}).
\item Note that we have a trade-off between more restrictive mixing conditions on the underlying process $(Z_n)_{n \in \Z}$ and more restrictive moment conditions on $X$. We make this explicit in Corollary~\ref{cor:biv_FCLT_abs_mixing}. 
\item For $r>1$, there can be an additional trade-off: In Corollary~\ref{cor:biv_FCLT_Lq_NED} we show that sufficient conditions for the $L_2$-NED of $\lvert X_n \rvert$ are a balance between more restrictive moment conditions than $(M_r)$ and less restrictive NED conditions on $(X_n)_{n \in \Z}$ (but more than $L_2$-NED).
\item The setting of Theorem~\ref{thm:biv_FCLT} includes, as special cases, the case of an underlying iid process $(Z_n)_{n \in \Z}$ as well as a causal representation of $(X_n)_{n \in \Z}$ with respect to $(Z_n)_{n \in \Z}$ - as used in many econometric time-series models; see Section~\ref{sec:exp} for examples.
\end{enumerate}
\end{remark}
\vspace*{-0.3cm}
The result of Theorem~\ref{thm:biv_FCLT} holds also true if we consider an underlying absolutely regular process $(Z_n)_{n \in \Z}$, and slightly adapt the conditions:
\begin{corollary} \label{cor:biv_FCLT_abs_mixing}
Assume the setting and conditions of Theorem~\ref{thm:biv_FCLT} with the following two changes:
\begin{itemize}
\item Take $\beta= \min{(s-3, \kappa)}$ and assume that the moment condition $(M_{r \frac{\beta+\epsilon }{\beta-1}})$ holds for an $\epsilon>0$, instead of $(M_r)$.
\item (D1$^*$) (instead of (D1)): The process $X$ is $L_2$-NED with constants $\nu(k) =  O(k^{-s}), s>6$, on an absolutely regular process $(Z_n)_{n \in \Z}$ (instead of $\phi$-mixing) with the same order of mixing coefficient, i.e. $\beta(n)=O(n^{-\kappa})$ for some $\kappa >3$.\\[-5ex]
\end{itemize}
Then, the FCLT holds for $T_{n,r}(X)$ as in Theorem~\ref{thm:biv_FCLT}. 
\end{corollary}
Moreover, we can provide sufficient conditions to reduce the $L_2$-NED of $(\lvert X_n \rvert^r)_{n \in \Z}$ (condition (D2) in Theorem~\ref{thm:biv_FCLT} and Corollary~\ref{cor:biv_FCLT_abs_mixing} respectively) to $L_q$-NED of the process $(X_n)_{n \in \Z}$ itself, as follows:
\begin{corollary} \label{cor:biv_FCLT_Lq_NED}
Assume the setting and conditions of Theorem~\ref{thm:biv_FCLT} ($(Z_n)_{n \in \Z}$ $\phi$-mixing) or Corollary~\ref{cor:biv_FCLT_abs_mixing} ($(Z_n)_{n \in \Z}$ absolutely regular), replacing (D2) with the following assumptions: For any choice of $p,q \in (1, \infty)$ s.t. $\frac{1}{p}+\frac{1}{q} = 1$, suppose that
\begin{itemize}
\item $\E[ \lvert X_0 \rvert^{p(2r-1)} ] < \infty$,
\item $(X_n)_{n \in \Z}$ is $L_q$-NED with constants $\nu(k)$ of the same order as in Theorem~\ref{thm:biv_FCLT} or Corollary~\ref{cor:biv_FCLT_abs_mixing}, i.e. $O(k^{-s}),s>6$. 
\end{itemize}
Then, the FCLT holds for $T_{n,r}(X)$ as in Theorem~\ref{thm:biv_FCLT} or Corollary~\ref{cor:biv_FCLT_abs_mixing}, respectively.  
\end{corollary}
Note the following two extreme cases for choices of $p$ and $q$ in Corollary~\ref{cor:biv_FCLT_Lq_NED}: Choosing $p=q=2$, we get the most restrictive moment condition ($\E[ \lvert X_0 \rvert^{4r-2}] < \infty$) in combination with $L_2$-NED of $(X_n)_{n \in \Z}$, while, when choosing $p=\frac{2r}{2r-1}$, we get the least restrictive moment condition ($\E[ \lvert X_0 \rvert^{2r}] < \infty$) but needing $L_{2r}$-NED of $(X_n)_{n \in \Z}$ at the same time. Only for $r=1$, these two cases coincide. 

Let us now turn to two special cases of Theorem~\ref{thm:biv_FCLT} presented in Corollaries~\ref{cor:biv_CLT} and~\ref{cor:biv_CLT_iid} and given for sake of completeness. Choosing $t=1$ in Theorem~\ref{thm:biv_FCLT} provides the usual CLT, stated in Corollary~\ref{cor:biv_CLT}:
\begin{corollary} \label{cor:biv_CLT} (CLT)
Under the same conditions as in Theorem~\ref{thm:biv_FCLT},
the joint behaviour of the sample quantile $q_n(p)$ (for $p \in (0,1)$) and the $r$-th absolute centred sample moment $\hat{m}(X,n,r)$, is asymptotically bivariate normal:
\begin{equation} 
\sqrt{n} \, \begin{pmatrix} q_n (p) - q_X(p) \\ \hat{m}(X,n,r)  - m(X,r) \end{pmatrix} \; \underset{n\to\infty}{\overset{d}{\longrightarrow}} \; \mathcal{N}(0, \Gamma^{(r)}), 
\end{equation}
where the asymptotic covariance matrix $\displaystyle \Gamma^{(r)}=(\Gamma_{ij}^{(r)}, 1\le i,j\le 2)$ is defined in Theorem~\ref{thm:biv_FCLT}.
\end{corollary}
We can also recover the CLT between the sample quantile and the $r$-th absolute centred sample moment stated in the iid case; see \cite{Brautigam22}, Theorem~4.1, which, of course, requires less restrictive conditions.
\begin{corollary} \label{cor:biv_CLT_iid} (Theorem~4.1 in \cite{Brautigam22}).
Consider the series of iid rv's $(X_n)_{n \in \Z}$ with parent rv $X$. 
Assume that $X$ satisfies $(M_r)$ and both conditions $(C_2^{~'})$ and $(C_1^{+})$ in a neighbourhood of $q_X(p)$. 
Additionally, for $r=1$, assume $(C_0)$ at $\mu$.
Then, the joint behaviour of the sample quantile $q_n(p)$ (for $p \in (0,1)$) and the $r$-th absolute centred sample moment $\hat{m}(X,n,r)$, is asymptotically bivariate normal:
\begin{equation} 
\sqrt{n} \, \begin{pmatrix} q_n (p) - q_X(p) \\ \hat{m}(X,n,r)  - m(X,r) \end{pmatrix} \; \underset{n\to\infty}{\overset{d}{\longrightarrow}} \; \mathcal{N}(0, \Gamma^{(r)}), 
\end{equation}
where the asymptotic covariance matrix $\displaystyle \Gamma^{(r)}=(\Gamma_{ij}^{(r)}, 1\le i,j\le 2)$ simplifies to
{\small
\begin{align*}
\Gamma_{11}^{(r)}&= \frac{p(1-p)}{f_X^2(q_X(p))}; \\ 
\Gamma_{22}^{(r)} &= r^2 \E[ X_0^{r-1} \sgn(X_0 )^r]^2 \sigma^2 + \Var(\lvert X_0  \rvert^r) - 2r \E[ X_0^{r-1} \sgn(X_0)^r] \Cov(X_0,\lvert X_0  \rvert);
\\  \Gamma_{12}^{(r)}&= \Gamma_{21}^{(r)} = \frac{1}{f_X(q_X(p))} \left( r \E[ X_0^{r-1} \sgn(X_0)^r] \Cov(\1_{(X_0 \leq q_X(p))},X_0) - \Cov(\1_{(X_0 \leq q_X(p))},\lvert X_0 \rvert^r) \right).
\end{align*}
}
\end{corollary}

{\it Idea of the proof -}
Let us briefly sketch in three main steps the proof of Theorem~\ref{thm:biv_FCLT} and, analogously, of Corollary~\ref{cor:biv_FCLT_abs_mixing}, which is developed in Section~\ref{sec:proofs}. 
\begin{itemize}
  \item First, we represent the sample quantile and $r$-th absolute sample moment as sums of functionals of the process $(X_n)$, so that $T_{n,r}(X)$ has the following representation, based on $f(X_j)_{j\in \N}$, for a measurable function $f: \mathbb{R} \rightarrow \mathbb{R}^2$, with $\E[f(X_j)] =0$ and $\| f(X_j) \|_2 < \infty$, $\forall j$,
  \begin{equation*} 
  T_{n,r}(X) = \frac{1}{n} \sum_{j=1}^n f(X_j) + o_P(1/\sqrt{n})\qquad\text{as $n \rightarrow \infty$.}
  \end{equation*}
For the sample quantile, we use its Bahadur representation, using a version 
given in \cite[Theorem 1]{Wendler11}. For that, we check that the NED, moments and mixing of the process are satisfied. These conditions are needed for approximating the sample quantile by an iid sample (and showing that the rest is asymptotically negligible). This approximation idea (coupling technique) dates back to the 70's with Sen (see e.g. \cite{Wu05} for a brief historical review on the Bahadur representation).
Next, we prove, under some conditions, a corresponding asymptotic representation for the $r$-th absolute centred sample moment, extending results from~\cite{Segers14} and \cite{Brautigam22}. Here, more work is involved as we consider not only the case when the mean is known but also an unknown mean. 
As such a representation is of interest on its own, we state it as a separate result, in Proposition~\ref{prop:Abs_central_mom_asympt_garch}. 
  \item Next, we have to ensure that each of the components of $f(X_j)_{j \in \N}$ fulfil the conditions needed for a multivariate FCLT.
 Those include moment conditions on $f(X_j)$, for any $j$, $L_2$-NED of $f(X_j)$, as well as summability conditions of the mixing coefficients and constants $\nu(k)$. 
 In particular, we show that the $L_2$-NED condition on $f(X_j)$ can be reduced to the $L_2$-NED condition for the processes $(X_j)_{j \in \Z}$ and $\left( \lvert X_j \rvert^r \right)_{j \in \Z}$.
  \item Finally, combining these results, we show that all conditions are met to apply a multivariate FCLT given in \cite[Theorem~3.2]{DeJong00} (a multivariate version of Theorem 1.2 in \cite{Davidson02}) that we present, adapted to our situation, in Theorem~\ref{thm:mult_FCLT}. 
\end{itemize}
Note that, for the FCLT to hold for the very general class of stationary processes considered in this study, we will have to separately analyse the cases of $\phi$-mixing (with moment condition $(M_r)$) and absolute regularity (with moment condition $(M_{r \frac{\beta+\epsilon}{\beta-1}})$) for $(Z_n)$. In each of the two cases, we need to introduce a different set of conditions for the FCLT to hold for $X_n=f(Z_n)$, making sure that the mixing conditions on $(X_n)$ are inherited from those on $(Z_n)$. The $\phi$-mixing case is done in the proof of Theorem~\ref{thm:biv_FCLT}, the absolute regularity one in the proof of Corollary~\ref{cor:biv_FCLT_abs_mixing}. 

\section{Application to augmented GARCH and ARMA processes}  \label{sec:exp}

We focus in this section on two classes of processes widely used in application, which fall within the framework of Theorem~\ref{thm:biv_FCLT}. By doing so, we provide new results when considering joint standard estimators for such processes. 
In Section~\ref{ssec:GARCH} we consider the broad family of augmented GARCH($p$, $q$) processes with iid innovations, while, in Section~\ref{ssec:ARMA}, ARMA processes with dependent (absolutely regular) innovations; we also discuss an example, the ARMA-GARCH process.
For the interested reader, we provide in the \hyperref[appn]{Appendix} 
further specific examples of processes that fall in one of the two mentioned families, with conditions on the moments and parameters of the process (for the presented results to hold) stated explicitly for convenience.

\subsection{Bivariate FCLT for augmented GARCH($p$,$q$) processes} \label{ssec:GARCH}

Since the introduction of the ARCH and GARCH processes in the seminal papers by Engle, \cite{Engle82}, and Bollerslev, \cite{Bollerslev86}, respectively, various GARCH modifications and extensions have been proposed and their statistical properties analysed (see e.g. \cite{Bollerslev08} for an (G)ARCH glossary). 
Many such well-known GARCH processes can be seen as special cases of the class of augmented GARCH($p$, $q$) processes, established by Duan in \cite{Duan97}.
An augmented GARCH($p$,$q$) process $X=(X_t)_{t \in \Z}$ satisfies, for integers $p \geq 1$ and $q\geq 0$,
\begin{equation}\label{eq:augm_GARCH_pq}
X_t = \sigma_t ~\epsilon_t 
\quad \text{with~} \quad \Lambda(\sigma_t^2) = \sum_{i=1}^p g_i (\epsilon_{t-i}) + \sum_{j=1}^q c_j (\epsilon_{t-j}) \Lambda(\sigma_{t-j}^2),
\end{equation}
where $(\epsilon_t)$ is a series of iid rv's with mean $0$ and variance $1$, $\sigma_t^2 = \Var(X_t)$ and $\Lambda, g_i, c_j, i=1,...,p, j=1,...,q$, are real-valued measurable functions. 
Also, as in \cite{Lee14}, we restrict the choice of $\Lambda$ to the so-called group of either polynomial GARCH($p$,$q$) or exponential GARCH($p$,$q$) processes (see Figure~\ref{fig:GARCH_nested} in the \hyperref[appn]{Appendix} for a representation of the hierarchy of these processes):
\[ (Lee) \quad \quad \Lambda(x) = x^{\delta}, \text{~for some~} \delta >0, \qquad \text{~or~} \qquad \Lambda(x) = \log(x).\]
Clearly, for a strictly stationary solution to~\eqref{eq:augm_GARCH_pq} to exist, the functions $\Lambda, g_i, c_j$ as well as the innovation process $(\epsilon_t)_{t \in \Z}$ have to fulfil some regularity conditions (see e.g. \cite{Lee14}, Lemma 1).
Alike, for the bivariate FCLT to hold, certain conditions on the functions $g_i, c_j, i=1,...,p, j=1,...,q$, of the augmented GARCH($p$,$q$) process of the $(Lee)$ family are needed, namely: Positivity of the functions used, (A), and boundedness in $L_r$-norm for either the polynomial GARCH, $(P_r)$, or exponential GARCH, $(L_r)$, respectively, for a given integer $r>0$,
\begin{align*}
&(A) && g_i \geq 0, c_j \geq 0, i=1,...,p,~j=1,...,q,
\\ &(P_r) && \sum_{i=1}^p \| g_i(\epsilon_0) \|_r < \infty, \quad \sum_{j=1}^q \| c_j(\epsilon_0) \|_r < 1, 
\\ &(L_r) &&  \E\left[ \exp\left(4r \sum_{i=1}^p \lvert g_i(\epsilon_0) \rvert^2\right)\right] < \infty, \quad \sum_{j=1}^q \lvert c_j(\epsilon_0) \rvert < 1. 
\end{align*}
Note that Condition $(L_r)$ requires the $c_j$ to be bounded functions.

\begin{remark}
By construction, from \eqref{eq:augm_GARCH_pq}, $\sigma_t$ and $\epsilon_t$ are independent (and $\sigma_t$ is a functional of $(\epsilon_{t-j})_{j=1}^{\infty}$). Thus, the conditions on the moments, distribution and density, could be formulated in terms of $\epsilon_t$ only. At the same time, this might impose some conditions on the functions $g_i, c_j, i=1,...,p, j=1,...,q$ (which might not be covered by $(A)$, $(P_r)$ or $(L_r)$). Thus, we keep the conditions on $(X_t)_{t \in \Z}$ as such, even if they might not be minimal.
\end{remark}
Now, let us explain why and under which circumstances Theorem~\ref{thm:biv_FCLT} holds. 
It has been shown in the literature under which conditions the class of augmented GARCH($p$, $q$) processes fulfills the $L_2$-NED. 
More precisely, the conditions of geometric $L_2$-NED on $(X_0)$ and $(\lvert X_0 \rvert^r)$ are satisfied, on the one hand in the polynomial case under $(M_r), (A)$ and $(P_{\max(1,r/\delta)})$ via Corollary 2 in \cite{Lee14}, on the other hand in the exponential case under $(M_r), (A)$ and $(L_r)$ via Corollary 3 in \cite{Lee14}. 
This can be directly used to reframe the general FCLT given in Theorem~\ref{thm:biv_FCLT} for the class of augmented GARCH($p$, $q$) processes: 
\\ 
$(Lee), (A)$ and $P_{max(1, r/\delta)}$ for polynomial, $(L_r)$ for exponential GARCH models respectively, are sufficient conditions for (D1) and (D2) in Theorem~\ref{thm:biv_FCLT} to hold. Thus, Theorem~\ref{thm:biv_FCLT} translates as follows:
\begin{corollary}\label{thm:augm_GARCH_pq_asympt_qn}
Consider an augmented GARCH($p$,$q$) process $X$ as defined in~\eqref{eq:augm_GARCH_pq}, which satisfies: 
\vspace{-2ex}
\begin{itemize}
\item $(M_r)$, both conditions $(C_2^{~'}), (C_1^{+})$ at $q_X(p)$, and $(C_0)$ at $0$ for $r=1$
\item $(Lee)$, $(A)$, and either $(P_{max(1,r/\delta)})$ for $X$ belonging to the group of polynomial GARCH, or $(L_r)$ for the group of exponential GARCH
\end{itemize}
\vspace{-2ex}
Then, for $t \in [0,1]$, as $n\to\infty$, we obtain the FCLT 
\[  \sqrt{n}~t ~ \begin{pmatrix} q_{[nt]}(p) -q_X(p) \\ \hat{m}(X,[nt],r) -m(X,r) \end{pmatrix}  \overset{D_2[0,1]}{\rightarrow}  \textbf{W}_{\Gamma^{(r)}} (t), \]
where the 2-dimensional Brownian motion, $(\textbf{W}_{\Gamma^{(r)}}(t))_{t \in [0,1]}$, has the same covariance matrix as in Theorem~\ref{thm:biv_FCLT}.
\end{corollary}

\subsection{Bivariate FCLT for ARMA($p$,$q$) processes}  \label{ssec:ARMA}

Similarly to the family of GARCH processes, ARMA (AutoRegressive - Moving Average) processes are widely used in applications (e.g. for financial time series).
While GARCH models specify a structure of the conditional variance of the process, ARMA processes specify the conditional mean. \\
Recall that a general ARMA($p$,$q$) process $X=(X_t)_{t \in \Z}$, for integers $p \geq 1$ and $q\geq 0$, is a stationary process defined by
\begin{equation}\label{eq:ARMA_pq}
\Phi(B) X_t=\Theta(B) \epsilon_t
\end{equation}
where $(\epsilon_t)$ is, in its most general form, a series of rv's with mean $0$ and finite variance, 
the backward operator $B$ denotes the application $B(X_t)=X_{t-1}$, 
and   $\Phi$ and $\Theta$ are polynomials of order $p$ and $q$ respectively, 
defined by
$\Phi(z)=1+\phi_1z+\ldots + \phi_pz^p$ ($\phi_p \ne 0$) and
$\Theta(z)=1+\theta_1z+\ldots +\theta_qz^q$ ($\theta_q \ne 0$),
such that $\Phi$ and $\Theta$ do not have any common root (to have a unique solution to \eqref{eq:ARMA_pq}).
A necessary and sufficient condition for $X$ to be causal, {\it i.e.} 
$\displaystyle X_t=\sum_{j=0}^{\infty}\psi_j\epsilon_{t-j}, \, t\in\Z$, with $\displaystyle  \sum_{j=0}^{\infty}|\psi_j|<\infty$, is:
\begin{equation}\label{eq:causal}
 \Phi(z)\ne 0, \forall z\in\C \mbox{  s.t.  } |z|\le 1.
\end{equation}

As already mentioned, various specifications of ARMA processes exist; we refer to \cite{Brockwell91}, \cite{Box15} and the survey article \cite{Holan10} for further references and examples. 
While the simplest case assumes iid innovations $(\epsilon_t)$, mixing innovations have been introduced as they allow for broader applications. We can consider this general case thanks to the setup of Theorem~\ref{thm:biv_FCLT} (and Corollary~\ref{cor:biv_FCLT_abs_mixing} respectively).

Note that the geometric $L_2$-NED of ARMA($p$, $q$) processes with mixing innovations can be directly deduced from Theorem 3.1 in \cite{Qiu11} under the causality condition \eqref{eq:causal}. 
But, in contrast to the case of augmented GARCH($p$,$q$) processes, we do not have results on the $L_2$-NED of $(\lvert X_0 \rvert^r)$. So we establish this property in Lemma~\ref{lemma:ARMA_NED_Xp} (see Section~\ref{sec:aux_res}), showing that the causality condition \eqref{eq:causal} is still sufficient. 
With these informations at hand, we can replace conditions (D1) and (D2) of Theorem~\ref{thm:biv_FCLT} by specific conditions for the class of ARMA($p$, $q$) processes with mixing innovations and state the following: 
\begin{corollary}\label{cor:ARMA_pq_asympt_qn}
Consider a causal ARMA($p$,$q$) process $X$ as defined in~\eqref{eq:ARMA_pq}, such that:
\vspace{-2ex}
\begin{itemize}
\item Both conditions $(C_2^{~'}), (C_1^{+})$ at $q_X(p)$ hold, as well as $(C_0)$ at $0$ for $r=1$
\item $X$ has either $\phi$-mixing innovations (denoted $X\sim$ARMA-$\phi$), or absolutely regular innovations ($X\sim$ARMA-$\beta$) - in both cases with mixing coefficient of order $O(n^{-\kappa})$, for some $\kappa>3$
\item  $X$ satisfies $(M_r)$ for $X$ ARMA-$\phi$, or $(M_\frac{r \kappa + \epsilon}{\kappa-1})$, for any $\epsilon>0$ for ARMA-$\beta$.
\end{itemize}
\vspace{-2ex}
Then, for $t \in [0,1]$, as $n\to\infty$, 
\vspace{-2ex}
\[  \sqrt{n}~t ~ \begin{pmatrix} q_{[nt]}(p) -q_X(p) \\ \hat{m}(X,[nt],r) -m(X,r) \end{pmatrix}  \overset{D_2[0,1]}{\rightarrow}  \textbf{W}_{\Gamma^{(r)}} (t), \]
\vspace{-2ex}
where $(\textbf{W}_{\Gamma^{(r)}}(t))_{t \in [0,1]}$, the 2-dimensional Brownian motion, has the same covariance matrix as in Theorem~\ref{thm:biv_FCLT}.
\end{corollary}
Among ARMA processes with mixing innovations, the ARMA-GARCH process is a widespread example in applications (see e.g. \cite{Song18, Spierdijk16, Hoga19}).
\begin{example}[ARMA($p$, $q$)-GARCH($r$,$s$) process (for $p\geq 1, q\geq0, r\geq 1, s \geq0$)] \label{exp:ARMA-GARCH}~\\
A general ARMA($p$, $q$)-GARCH($r$,$s$) $(X_t)_{t \in \Z}$  is defined as follows: 
$$
X_t = \sum_{i=1}^p \phi_i y_{t-i} + \sum_{i=1}^q \theta_i \epsilon_{t-i} + \epsilon_t,
\;\text{where}\;\epsilon_t = \eta_t \sigma_t \;\text{and}\;  \sigma_t^2 = \alpha_0 + \sum_{i=1}^r \alpha_i \epsilon_{t-i}^2 + \sum_{i=1}^s \beta_i \sigma_{t-i}^2,
$$
 for $(\eta_t)$ an iid series with mean $0$ and variance $1$.

Let us discuss the conditions of Corollary~\ref{cor:ARMA_pq_asympt_qn} applied to this class of processes. 
\begin{itemize}
\item Causality: For an ARMA($p$,$q$)-GARCH($r$, $s$) to be causal, \eqref{eq:causal} needs to be fulfilled
\item Conditions on the marginal distribution: $(C_2^{~'}), (C_1^{+})$ at $q_X(p)$ as well as $(C_0)$ at $0$ for $r=1$
\item Mixing innovations: The GARCH($r$,$s$) process is known to be absolutely regular with geometric rate as long as it is strictly stationary, $\eta_0$ is absolutely continuous with Lebesgue density being strictly positive in a neighbourhood of 0, and $\E[\lvert \eta_0 \rvert^{t}]$ for some $t \in (0, \infty)$; see \cite[Theorem~8]{Lindner09} (the original result going back to \cite{Boussama98}).\\ 
A necessary condition for the strict stationarity of the GARCH($p$,$q$) process is known in the literature; see e.g. \cite[Theorem 1.3]{Bougerol92}. A sufficient condition, easier to verify in practice, is: $\E[\eta_0^2] \sum_{i=1}^r \alpha_i + \sum_{i=1}^s \beta_i < 1$ for $\E[\eta_0^2]< \infty$ (see \cite{Bollerslev86}). 
\item Moment conditions: As the innovations are absolutely regular with geometric rate, the corresponding moment condition is $(M_{\frac{r \kappa+ \epsilon}{\kappa-1}})$ for any $\kappa, \epsilon>0$. 
Necessary conditions for the existence of such moments depend on the specifications of the ARMA process and can be found e.g. in \cite[Theorem 2.2]{Ling03}.
In practice, a sufficient moment condition would be $(M_{r+1/2})$. 
\end{itemize}
\end{example}

\section{Auxiliary Results} \label{sec:aux_res}

In this section we present two different types of results. 
First, we provide in Proposition~\ref{prop:Abs_central_mom_asympt_garch} an asymptotic representation of the $r$-th absolute centred sample moment $\hat{m}(X,n,r)= \frac{1}{n} \sum_{i=1}^n \lvert X_i - \bar{X}_n \rvert^r$, for any integer $r\geq 1$. 
Then, in Lemmas~\ref{lemma:L2NED_bd_fct},~ \ref{lemma:NED_for_polys_general} and ~\ref{cor:NED_for_polys}, we propose sufficient conditions such that the $L_2$-NED property of $(X_n)_{n \in \Z}$ is inherited for certain bounded and unfunctionals $(f(X_n))_{n \in \Z}$. We use them to establish the $L_2$-NED of $(\1_{\left( X_n \leq q_X(p) \right)})_{n \in \Z}$ and $(\lvert X_n \rvert^r)_{n \in \Z}, r \in \N$, needed in the proof of Theorem~\ref{thm:biv_FCLT}.  
Additionally, Lemma~\ref{lemma:ARMA_NED_Xp} treats the specific case of the $L_2$-NED of $(\lvert (X_n) \rvert^r)_{n \in \Z}$ for any integer $r$ for ARMA($p$,$q$) processes. 
Proofs of these results are deferred to Section~\ref{ssec:aux_res_proofs}.

\subsection{Representation for the $r$-th absolute centred sample moment}

Let us state the asymptotic relation between
the $r$-th absolute centred sample moment with known and unknown mean, respectively. This enables us to compute the asymptotics of $\hat{m}(X,n,r)$ (given that the needed moments exist). 
\begin{proposition} \label{prop:Abs_central_mom_asympt_garch}
Consider a stationary and ergodic process $(X_n)_{n \geq 0}$ that has `short-memory', i.e. \\
 $ \sum_{i=0}^{\infty} \lvert \Cov(X_0, X_i) \rvert < \infty$. 
Then, for any integer $r\geq 1$, assuming that the $r$-th moment of $X$ exists and $(C_0)$ holds at $\mu$ for $r=1$, we have the following asymptotics, as $n \rightarrow \infty$, 
{\small
\begin{eqnarray} \label{eq:Abs_central_mom_asympt_garch}
&&\sqrt{n} \left(\frac{1}{n} \sum_{i=1}^n \lvert X_i - \bar{X}_n \rvert ^r \right) = \nonumber\\
&& \sqrt{n} \left(\frac{1}{n} \sum_{i=1}^n \lvert X_i - \mu \rvert^r \right) - r \sqrt{n} (\bar{X}_n - \mu) \E[ (X_0 - \mu)^{r-1} \sgn(X_0 - \mu)^r] + o_P(1).
\end{eqnarray}
}
\end{proposition}

\subsection{$L_2$-NED of functionals of $(X_n)_{n \in \Z}$}

Here, we question under which conditions some bounded and unbounded functionals, as e.g. $(\1_{\left( X_n \leq x \right)})_{n \in \Z}$ and $(\lvert X_n \rvert^r)_{n \in \Z}, r \in \N$, inherit the $L_2$-NED of the process $X=(X_n)_{n\in\Z}$.

For bounded functionals, we can adapt the result proved for the $L_1$-NED case (see \cite{Wendler11}, Lemma~3.5) to the $L_2$-NED.  
For this, we introduce the `$p$-variation` condition that dates back to \cite{Denker86}, also 
used in \cite[Definition 1.4]{Wendler11}. 
As already noticed there, 
the $p$-variation condition is similar to the notion of $p$-continuity of \cite[Definition 2.10]{Borovkova01}.
\begin{definition} \label{def:p_variation}
Let $(X_n)_{n \in \Z}$ be a stationary process. 
For $p>0$, a function $g: \R \rightarrow \R$ satisfies the $p$-variation condition (with respect to the distribution of $X_0$), if there exists a constant $c$ such that for any $\epsilon>0$
\[ \E[ \sup_{x \in \R: \lvert x - X_0 \rvert \leq \epsilon} \lvert g(x) - g(X_0) \rvert^p ] \leq  c~\epsilon^p.\]
\end{definition}
Then, we can adapt Lemma~3.5 in~\cite{Wendler11} to the $L_2$-NED case. 
\begin{lemma} \label{lemma:L2NED_bd_fct}
Let $(X_n)_{n \in \Z}$ be $L_2$-NED with constants $\nu(k), k \in \N$, on a stationary process $(Z_n)_{n \in \Z}$.
    \begin{itemize}
    \item[(i)] Let $g$ be a function bounded by $K$ such that $g$ satisfies the 2-variation condition with constant $L$. Then,
    \begin{equation}\label{eq:CS-L2Ned}
        (g(X_n))_{n \in \Z} \quad\text{is} \quad L_2-\text{NED with constants} \quad\sqrt{\left(  L + 4K^2 \right) \nu(k)},\; k \in \N. 
    \end{equation}
    \item[(ii)] For the specific case of $g$ being the indicator function $g(x) := \1_{\left( x \leq t \right)}$ ($t \in \R$), the result \eqref{eq:CS-L2Ned} holds under the 1-variation condition (instead of the 2-variation). 
    \end{itemize}
\end{lemma}

Note that the result (ii) will be needed in the proof of Theorem~\ref{thm:biv_FCLT}.
\begin{remark}
    Recall that a sufficient condition for the indicator function to satisfy the 1-variation condition with respect to the distribution of $X_0$, is the Lipschitz-continuity of the distribution function of $X_0$, as shown in \cite[Example 1.5]{Wendler11}.
\end{remark}
Turning to unbounded functions, rather than considering $p$-variation (or $p$-continuity), we assume a certain geometry on these functions, namely convexity.
\begin{lemma} \label{lemma:NED_for_polys_general}
Let a process $X = (X_n)_{n \in \Z}$ be $L_2$-NED with constants $\nu(k)$ on a stationary process $ (Z_n)_{ n \in \Z}$. 
Consider a positive, convex function $f$ with derivative $f'$ such that, on $\R^{+}$, $f\times f'$ is convex and positive, and $(f\times f')^p$ is convex with $p>1$.
Then, a sufficient condition for the $L_2$-NED of $(f(X_n))_{n \in \Z}$ with constants $\tilde{\nu}(k)= O(\nu^{1/2}(k))$ is :
{\small
\[ \E[(f(\lvert X_n \rvert) f'(\lvert X_n \rvert))^p] < \infty \text{~and~} X \text{~is~} L_q-\text{NED for $q>1$ s.t.} \; \frac{1}{p} + \frac{1}{q} = 1, \text{ with ~} \tilde{\nu}(k) = O(\nu(k)).\]
}
\end{lemma}
Applying Lemma~\ref{lemma:NED_for_polys_general} for the function $f(\cdot) = \lvert \cdot \rvert^p$, $p \in \Z$, we obtain:
\begin{lemma} \label{cor:NED_for_polys}
Let a process $X = (X_n)_{n \in \Z}$ be $L_2$-NED with constants $\nu(k)$ on a stationary process $ (Z_n)_{ n \in \Z}$. Then, for a given integer $r \geq 1$, a sufficient condition for the $L_2$-NED of $(\lvert X_n \rvert^r)_{n \in \Z}$ (with constants $\tilde{\nu}(k)$ of order $\tilde{\nu}(k) = O(\nu(k)^{1/2})$) is, for any choice of $p,q \in (1, \infty)$ s.t. $\frac{1}{p} + \frac{1}{q} = 1$, 
\[ \E[ \lvert X_n \rvert^{p (2r-1)}] < \infty, \text{~and~} L_q-\text{NED of~} X \text{~with constants of order~} O(\nu(k)).\]
\end{lemma}
As an application, let us consider the class of ARMA($p$, $q$) processes, introduced in Section~\ref{ssec:ARMA}.
\begin{lemma} \label{lemma:ARMA_NED_Xp}
Let $X$ be a causal ARMA($p$, $q$) process as defined in \eqref{eq:ARMA_pq}. Then, for any integer $r>0$, $(\lvert X_0 \rvert^r)$ is geometrically $L_2$-NED if $\E[\lvert X_0 \rvert^{2r}] < \infty$. 
\end{lemma}

\section{Proofs} \label{sec:proofs}

The main theorem is proven in Section~\ref{ssec:main_proof}, while the corollaries in Section~\ref{ssec:cor_proofs}.
Proofs of the auxiliary results (presented in Section~\ref{sec:aux_res}) are given in Section~\ref{ssec:aux_res_proofs}. 
 
\subsection{Proof of Theorem~\ref{thm:biv_FCLT}} \label{ssec:main_proof}

As described in Section~\ref{sec:asympt_results}, the proof consists of three main steps. The first one, more involved, is split into two parts.  

{\sf Step 1: Asymptotic representations for the two sample estimators} \\[1ex]
$\mathbf{1a):}$ {\sf A Bahadur representation for the sample quantile}\\[.7ex]
As explained in Section~\ref{sec:asympt_results} the main idea in the Bahadur representation is to approximate the sample quantile by an equivalent estimator based on iid rv, i.e. showing that the dependence is asymptotically negligible. 
The $L_p$-NED of $(X_n)$ allows to approximate these by functionals of finitely many $Z_n$, and the $p$-variation condition ensures that this also holds for $f(X_n)$. Coupling techniques are then used to show that short-range dependent variables have the same behavior as iid ones (for this, mixing conditions are also necessary). 
The choice of mixing conditions in Theorem~\ref{thm:biv_FCLT} on the underlying process comes from the use of the Bahadur representation for NED processes as given in \cite[Theorem~1]{Wendler11}, for which we need to verify the following conditions. \\
$(i)$ Let $g(x,t) := \1_{(x \leq t)}$. It is straightforward to check that $g$ is non-negative, bounded, measurable, and non-decreasing in the second variable. The function $g$ also satisfies the variation condition uniformly in some neighbourhood of $q_X(p)$ if it is Lipschitz-continuous. 
The latter follows from condition $(C_2^{~'})$ in a neighbourhood of $q_X(p)$.\\ 
$(ii)$ The differentiability of $\E[g(X,t)]= F_X(t)$ and positivity of its derivative at $t=q_X(p)$ are given by condition $(C_1^{+})$ at $q_X(p)$. \\
$(iii)$ The condition 
\[ \lvert F_X(x)-F_X(q_X(p)) - f_X(q_X(p)) (x-q_X(p)) \rvert = o \left( \lvert x - q_X(p)\rvert^{3/2} \right) \quad \text{as~} x\rightarrow q_X(p) \] 
is fulfilled as, by Assumption $(C_2^{~'})$, $F_X$ is twice differentiable in $q_X(p)$.\\ 
$(iv)$ As $(Z_n)_{n \in \Z}$ is stationary and $\phi$-mixing (Assumption $(D1)$), $(X_n)_{n \in \Z}$, being a function of $(Z_n)_{n \in \Z}$, is also stationary and ergodic. \\
$(v)$ Lastly, the process exactly fulfills the conditions on the mixing rate on the underlying process $(Z_n)_{n \in \Z}$  by assumption $(D1)$. Indeed taking 
$\beta := \min{(s-3, \kappa)}>3$, then, $\beta(n) = O(n^{-\beta})$ and $\nu(k) = O(k^{-(\beta+3)})$ as $\beta>3$. 
Further, the $L_1$-NED is implied by the assumption of $L_2$-NED (Assumption $(D1)$) at the exact same rate:
\begin{equation} \label{eq:NED-condition}
\| X_n -  \E[X_n \vert \mathcal{F}_{n-k}^{n+k}] \|_1 \leq \| X_n - \E[X_n \vert \mathcal{F}_{n-k}^{n+k}] \|_2 = O(k^{-(\beta +3)}).
\end{equation}
Thus, we can use the version of the Bahadur representation given in \cite[Theorem~1]{Wendler11}, replacing (for our purposes) the exact remainder bound by $o_P(1/\sqrt{n})$, and write, as $n \rightarrow \infty$,
\begin{equation} \label{eq:Bahadur_qn_Wendler}
q_n(p)- q_X(p) + \frac{F_X(q_X(p))-F_n(q_X(p))}{f_X(q_X(p))} = o_P(1/\sqrt{n}).
\end{equation}

$\mathbf{1b):}$ {\sf Representation of the $r$-th absolute centred sample moment}\\[.7ex]
The representation is given in Proposition~\ref{prop:Abs_central_mom_asympt_garch} (which proof can be found in Section~\ref{ssec:aux_res_proofs}). So, we simply need to check that the respective conditions (stationarity, finite $2r$-th moment, ergodicity and short-memory) are fulfilled. Let us explain why these conditions hold. 
We have seen in $(iv)$ above that $(X_n)$ is stationary.
The moment condition is fulfilled by assumption due to $(M_r)$ and, for $r=1$, we have $(C_0)$ at $\mu$ by assumption too.
To prove the ergodicity and short-memory property, we use a classical CLT for functionals of $\phi$-mixing processes, \cite[Theorem 21.1]{Billingsley68}. It means to check that its conditions are fulfilled. 
  We have $\E[ \lvert X_0 \rvert^{2}] < \infty$, by $(M_r), r \geq 1$.
  Moreover, as, by assumption $(D1)$,  we have $L_2$-NED with constants of the order $O(k^{-s}), s>6$,  it holds for those constants that $\displaystyle \sum_{k=1}^{\infty} \nu(k) < \infty$.
  Finally, $\sum_{n=1}^{\infty} \beta(n)^{1/2} < \infty$, since via $(D1)$ the $\phi$-mixing rate is assumed to be of order ${\beta(n) = O(n^{-\kappa}), \kappa>3}$.

{\sf Step 2: Establishing the conditions needed for each component $f(X_i)$ in view of the FCLT}\\[1ex]
Following the representation~\eqref{eq:Abs_central_mom_asympt_garch} of $\hat{m}(X,n,r)$, we introduce tridimensional random vectors $u_j=(u_{j,l}, l=1,2,3)$, for $j \in \N$ (anticipating their use in Step 3 for the FCLT of $U_n (X):=\frac{1}{n} \sum_{j=1}^n u_j$), defined by
\[ u_j = \begin{pmatrix} X_j \\ \lvert X_j  \rvert^r - m(X,r) \\ \frac{p - \1_{(X_j \leq q_X(p))}}{f_X(q_X(p))}  \end{pmatrix}. \]
The idea is then to apply a multivariate FCLT, which we adapt from 
\cite[Theorem 3.1]{DeJong00} / \cite{Davidson02} (a multivariate version of Theorem 1.2) to our needs. Their result is not restricted to stationary processes, but, for us, this is sufficient. This is why our conditions, by stationarity, will hold uniformly in $t$. It is also the reason why we do not need, in the definition of $L_p$-NED, a constant depending on the time $t$ of the process. Let us present this FCLT here, so that it clearly states the required conditions that each random process $(u_{j,l})_{j\in\N}$, for $l=1,2,3$, respectively, has to satisfy, in view of its application.

\begin{theorem} \label{thm:mult_FCLT}
Consider a d-dimensional stationary random process $(u_j, j\in \N)$ 
where: 
\\(a) Each of the $d$ components is an $L_2$-NED process with constants $\nu(k) = O(k^{-\beta}), \beta >1/2$, with respect to the same (univariate) process $(V_t)_{t \in \Z}$, which is either $\alpha$-mixing of order $\alpha(n) = O(n^{-\tilde{s}})$ for $\tilde{s}>\check{s}/(\check{s}-2)$, $\check{s}>2$, or $\phi$-mixing of order $\phi(n) = O(n^{-\tilde{s}})$ for $\tilde{s}> \check{s}/(2\check{s}-2)$, $\check{s}\geq 2$; 
\\(b) For this choice of $\check{s}$, it must hold that $\| u_{j,l} \|_{\check{s}} < \infty, ~\forall j,l$;
\\(c) $\Var(\sum_{j=1}^n u_{j,l})/n := \sigma_{n,l}/n \rightarrow \sigma_l^2 >0$, $\forall l$, as $n\to\infty$.
\\Then, the series $\Gamma = \sum_{j \in \Z} \Cov( u_0, u_j)$ converges (coordinate wise) absolutely and a FCLT holds for $U_n := \frac{1}{n} \sum_{j=1}^n u_j$, {\it i.e.}
\[ \sqrt{n}t U_{[nt]} \overset{D_d[0,1]}{\rightarrow} W_{\Gamma}(t), \quad\mbox{as}\;n\to\infty,\]
where the convergence takes place in the $d$-dimensional Skorohod space $D_d[0,1]$ and  $(W_{\Gamma}(t), t \in [0,1])$ is a $d$-dimensional Brownian motion with covariance matrix $\Gamma$, {\it i.e.} it has mean 0 and $\Cov(W_{\Gamma}(u), W_{\Gamma}(t)) = \min(u,t) \Gamma$.
\end{theorem}
\vspace{-3ex}
\noindent Choosing $\check{s}=2$, 
let us show that all the assumptions hold for each of the components of $(u_j)$:
\\- Condition (a):
Let us first comment on the order of the $L_2$-NED constants. 
\\For $u_{j,1} = X_j $, by Assumption $(D1)$, it is of order $O(k^{-s})$.
\\For $u_{j,2} = \1_{(X_j \leq q_X(p))}$, by $(D1)$ in conjunction with Lemma~\ref{lemma:L2NED_bd_fct} ii), we have an order of $O(k^{-s/2})$. 
\\Finally, for $u_{j,3} = \lvert X_j \rvert^r$, it is $O(k^{-\gamma})$ by Assumption $(D2)$. 
\\ As $s>6$ in $(D1)$ and $\gamma>2$ in $(D2)$, the required order for  Theorem~\ref{thm:mult_FCLT} is fulfilled.\\
The mixing rate for $(u_{j,l})$, for $l=1,2,3$, is, by Assumption $(D1)$, $O(n^{-\kappa}), \kappa >3$. As we chose $\check{s}=2$, the required rate in Theorem~\ref{thm:mult_FCLT} is $O(n^{-\tilde{s}})$ for $\tilde{s}> \check{s}/(2\check{s}-2)=1$. Thus $\kappa>3$ fulfills this requirement.
\\- Condition (b):
We treat each component of $u_{j,l}, l=1,2,3$, separately. 
By assumption of $(M_r), r \geq 1$, it holds that $\E[\lvert X_n \rvert^s] < \infty$ (with $s=2$, as chosen). Further, as $\1_{(X_j \leq q_X(p))}$ is bounded,  $\E[\lvert \1_{(X_j \leq q_X(p))} \rvert^s] < \infty$ holds. Finally,  
$\E[\lvert X_n \rvert^{rs}] < \infty$ holds using again $(M_r), r \geq 1$.
\\- Condition (c):
For $(u_{j,1})$, the convergence of the variance of the partial sums was already shown in Step~1b) (using Theorem 21.1 of \cite{Billingsley68}). By the exact same argument, using once again Billingsley's theorem, (c) holds for $u_{j,2}$ and $u_{j,3}$. 

{\sf Step 3: Multivariate FCLT}\\
Having checked the conditions for the FCLT (Theorem~\ref{thm:mult_FCLT}) in Step 2, we can now apply a trivariate FCLT for $(u_j)_j$.
Using the Bahadur representation \eqref{eq:Bahadur_qn_Wendler} of the sample quantile (ignoring the rest term for the moment), we can state:
\begin{equation} \label{eq:asympt_MAD_trivariate_normal}
 \sqrt{n} \frac{1}{n} \sum_{j=1}^{[nt]} u_j =   \sqrt{n}~t \begin{pmatrix} \bar{X}_{[nt]}  \\ \frac{1}{[nt]} \sum_{j=1}^{[nt]} \lvert X_j \rvert^r - m(X,r) \\ \frac{p-F_{[nt]}(q_X(p))}{f_X(q_X(p))} \end{pmatrix} \overset{D_3[0,1]}{\rightarrow}  \textbf{W}_{\tilde{\Gamma}^{(r)}} (t) \quad \text{~as~} n \rightarrow \infty,
\end{equation}
where $\left(\textbf{W}_{\tilde{\Gamma}^{(r)}}(t), t \in [0,1]\right)$ is the three-dimensional Brownian motion with covariance matrix ${\tilde{\Gamma}^{(r)}} \in \R^{3\times 3}$, {\it i.e.} the components ${\tilde{\Gamma}^{(r)}}_{ij}, 1\leq i,j \leq 3$, satisfy the same dependence structure as for the random vector $(U,V,W)^T$ described in $(Co)$, with all series being absolutely convergent. 
By the multivariate Slutsky theorem, we can add $\begin{pmatrix} 0 \\ 0 \\ R_{[nt],p} \end{pmatrix}$ to the asymptotics in \eqref{eq:asympt_MAD_trivariate_normal} without changing the resulting distribution (since $\sqrt{n} R_{[nt],p} \overset{P}\rightarrow 0$ as $n\to\infty$). Hence, we obtain, as $n \rightarrow \infty$,
{\small
\begin{eqnarray} \label{eq:conv_semifinal}
\hspace*{-3ex} \sqrt{n}~t \begin{pmatrix} \bar{X}_{[nt]}  - \mu \\ \frac{1}{[nt]} \sum_{j=1}^{[nt]} \lvert X_j \rvert^r - m(X,r) \\ \frac{p-F_{[nt]}(q_X(p))}{f_X(q_X(p))} \end{pmatrix} +  \sqrt{n}~t \begin{pmatrix} 0 \\ 0 \\ R_{[nt],p} \end{pmatrix} &=&  \sqrt{n}~t \begin{pmatrix} \bar{X}_{[nt]}  - \mu \\ \frac{1}{[nt]} \sum_{j=1}^{[nt]} \lvert X_j \rvert^r - m(X,r) \\ q_{[nt]} (p) - q_X(p) \end{pmatrix}  \nonumber\\
&\overset{D_3[0,1]}{\rightarrow}&  \textbf{W}_{\tilde{\Gamma}^{(r)}} (t).
\end{eqnarray} 
}
Then, we apply to  \eqref{eq:conv_semifinal} the multivariate continuous mapping theorem using the function $f(x,y,z) \mapsto (ax+y, z)$ with $a= -r \E[(X-\mu)^{r-1} \sgn(X-\mu)^r]$. Further, by Slutsky's theorem, we can add to $ax+y$ a rest of $o_P(1/\sqrt{n})$ without changing the limiting distribution.
So, we  obtain, as $n \rightarrow \infty$,
\begin{align*} 
\sqrt{n}~t &\begin{pmatrix} a (\bar{X}_{[nt]}  - \mu) + \frac{1}{[nt]} \sum_{j=1}^{[nt]} \lvert X_j \rvert^r - m(X,r) + o_P(1/\sqrt{n}) \\ q_{[nt]}(p) - q_X(p) \end{pmatrix} 
\\ &=  \sqrt{n}~t \begin{pmatrix} \hat{m}(X,[nt],r) - m(X,r) \\ q_{[nt]}(p) - q_X(p) \end{pmatrix} \overset{D_2[0,1]}{\rightarrow}  \textbf{W}_{\Gamma^{(r)}} (t), \numberthis \label{eq:asymptot1_MAD}
\end{align*} 
where $\Gamma^{(r)}$ follows from the specifications of $\tilde{\Gamma}^{(r)}$ above and the continuous mapping theorem. \hfill $\Box$

\subsection{Proofs of Corollaries} \label{ssec:cor_proofs}

We start with the proof of Corollary~\ref{cor:biv_FCLT_abs_mixing}, where we need to show why Theorem~\ref{thm:biv_FCLT} also holds in the case of an underlying absolutely regular process. The proofs of corollaries~\ref{cor:biv_CLT}, ~\ref{cor:biv_CLT_iid} and~\ref{thm:augm_GARCH_pq_asympt_qn}, are a direct consequence of Theorem~\ref{thm:biv_FCLT}, so can be omitted. 
The proof of Corollary~\ref{cor:biv_FCLT_Lq_NED} mainly consists of establishing sufficient conditions for the $L_2$-NED of $(\lvert X_n \rvert^r)_{n \in \Z}$, which will be done in (the proof of) Lemma~\ref{cor:NED_for_polys} (proven in Section~\ref{ssec:aux_res_proofs}). We end with the proof of Corollary~\ref{cor:ARMA_pq_asympt_qn} for which the main work consists of proving the $L_2$-NED of $\lvert X_0 \rvert^r, r>0$, for ARMA processes.
%
\begin{proof}{\bf of Corollary~\ref{cor:biv_FCLT_abs_mixing}.}
The proof follows the one of Theorem~\ref{thm:biv_FCLT}, but we need to adapt (parts of) Steps~1b) and~2 to the setting of Corollary~\ref{cor:biv_FCLT_abs_mixing}, {\it i.e.} for an underlying absolute regular process. Thus, we revisit these two steps. \\
{\sf \small Step 1b): Representation of the $r$-th absolute centred sample moment - Conditions.}\\
Given the representation given in Proposition~\ref{prop:Abs_central_mom_asympt_garch}, we need to check that the respective conditions, stationarity, finite $2r$-th moment, ergodicity and short-memory, are fulfilled. 
Only the reasoning for the short-memory property differs from the one in Theorem~\ref{thm:biv_FCLT}:
To prove the ergodicity and short-memory property, we verify that the conditions for a CLT of $X_n$ are fulfilled. 
Since absolute regularity implies strong mixing at the same rate, we consider the CLT for functionals of strongly mixing processes; see \cite[Theorem 18.6.2]{Ibragimov71}. 
By choosing $\delta = 2 \frac{1+ \epsilon}{\beta-1}$ in that theorem, we check that the conditions stated there are fulfilled (recall that we defined in Corollary~\ref{cor:biv_FCLT_abs_mixing}, $\beta := \min{(s-3, \kappa)}$). We have $\E[ \lvert X_0 \rvert^{2+\delta}] = \E[ \lvert X_0 \rvert^{2+\frac{2+  2\epsilon}{\beta -1} }] < \infty$ by $(M_{r\frac{\beta + \epsilon}{\beta-1} })$ (as $r\geq 1$). Moreover,
$\sum_{k=1}^{\infty} \| X_0 - \E[X_0 \vert \mathcal{F}_{-k}^{+k}] \|_{(2+\delta)/(1+\delta)} = \sum_{k=1}^{\infty} \| X_0 - \E[X_0 \vert \mathcal{F}_{-k}^{+k}] \|_{1+ \frac{1}{1+\delta}} < \infty$ holds as it is bounded by $\sum_{k=1}^{\infty}  \| X_0 - \E[X_0 \vert \mathcal{F}_{-k}^{+k}]\|_2$ (since $\delta>0$), which is finite by the assumption of $L_2$-NED with rate $O(k^{-s}), s>6$, {\it i.e.} $(D1^*)$. Finally,
$\sum_{n=1}^{\infty} \beta(n)^{\delta / (2+\delta)} < \infty$ holds by construction (the choice of $\delta$ was made in a way that the sum is finite): $(D1^*)$ ensures that $\beta(n) = O(n^{-\kappa})$ and, by the choice of $\beta$ in the corollary, implies $\beta(n) = O(n^{-\beta})$. We assume, w.l.o.g., that $\beta(n) = n^{-\beta}$. 
We then have
\[ \sum_{n=1}^{\infty} \beta(n)^{\delta / (2+\delta)} =  \sum_{n=1}^{\infty} n^{ - \beta \frac{\frac{2+ 2 \epsilon}{\beta-1} }{2 + \frac{2+ 2 \epsilon}{\beta-1} }} \;\text{~and~}\quad \beta \frac{\frac{2+ 2 \epsilon}{\beta-1} }{2 + \frac{2+ 2 \epsilon}{\beta-1} } =  1 +\frac{\epsilon (\beta-1)}{\beta +\epsilon}. \] 
As $\beta:= \min{(s-3, \kappa)}>3$,  $\epsilon>0$, this quantity will always be bigger than $1$ and hence the infinite sum remains summable.
Hence Theorem~18.6.2 in \cite{Ibragimov71} applies in this case for the process $(X_n)_{n \in \Z}$.
\\
{\sf\small Step 2: Conditions for applying the FCLT (Theorem~\ref{thm:mult_FCLT})}\\
In Theorem~\ref{thm:biv_FCLT}, we used Theorem~\ref{thm:mult_FCLT} as multivariate FCLT, which does not only cover the case of underlying $\phi$-mixing processes, but also strong mixing. It means that we can also use it here. Therefore, it comes back to verify Condition $(a^*)$, defined below, and Conditions (b) and (c) given in the proof of Theorem~\ref{thm:biv_FCLT}, Step 2.
\begin{itemize}
\item[$(a^*)$] $L_2$-NED process with constants $\nu(k) = O(k^{-\eta}), \eta >1/2$ (changed from $\beta$ to $\eta$, to avoid notational confusion) on a univariate process $(V_t)_{t \in \Z}$ (the same for all components), which, in this case, is $\alpha$-mixing of order $\alpha(n) = O(n^{-\tilde{s}})$ for any $\tilde{s}> s/(s-2)$ for a $s > 2$.
\end{itemize}
Note that, in comparison to the proof of Theorem~\ref{thm:biv_FCLT}, only the first condition has been adapted. Nevertheless, as the underlying mixing property is different, the proof of conditions (b) and (c) have also to be adapted.
Choosing $s=2\frac{\beta+\epsilon}{\beta-1}$ in Theorem~\ref{thm:mult_FCLT}, we can show that all the assumptions hold for each of the components of $(u_j)$.
\begin{itemize}
\item Condition $(a^*)$:
The order of the $L_2$-NED constants being the same as for the $\phi$-mixing case, see $(D1^*)$ and $(D2)$, the same arguments as in the proof of Theorem~\ref{thm:biv_FCLT} hold.
The mixing rate for $(u_{j,l}), l=1,...3,$ is, by assumption $(D1^*)$, $O(n^{-\kappa}), \kappa >3$, which implies, by definition of $\beta$, that it is also of order $O(n^{-\beta}), \beta>3$. Since we chose $s=2\frac{\beta+\epsilon}{\beta-1}$, we have $\frac{s}{s-2}= \beta \frac{1+ \epsilon/\beta}{1+\epsilon}$ (as $\epsilon>0,\, \beta>3$). The required rate in Theorem~\ref{thm:mult_FCLT} is $O(n^{-\tilde{s}})$ for a $\tilde{s}> s/(s-2)=\beta \frac{1+ \epsilon/\beta}{1+\epsilon} < \beta$, thus $\beta>3$ fulfills this requirement.
\item Condition (b):
 We treat each component of $(u_{j,l}), l=1,2,3$, separately.  By assumption of $(M_{r\frac{\beta+ \epsilon}{\beta-1} }), r \geq 1$, it holds that $\E[\lvert X_n \rvert^s] < \infty$. Further, as $\1_{(X_j \leq q_X(p))}$ is bounded,  $\E[\lvert \1_{\left(X_j \leq q_X(p)\right)} \rvert^s] < \infty$. Finally,  
$\E[\lvert X_n \rvert^{rs}] < \infty$, using again the assumption of $(M_{r\frac{\beta+ \epsilon}{\beta-1} })$, $r\geq 1$. 
\item Condition (c): 
For $(u_{j,1})$, the convergence of the variance of the partial sums was already shown in Step~1b) (using \cite[Theorem 18.6.2]{Ibragimov71}). By the same argument, with the same choice of $\delta$ as for $u_{j,1}$, it also follows for $u_{j,2}$ and $u_{j,3}$ (using once again Ibragimov's theorem).
\end{itemize}
\end{proof}
%
\begin{proof}{\bf of Corollary~\ref{cor:ARMA_pq_asympt_qn}.}
Comparing the conditions of Corollary~\ref{cor:ARMA_pq_asympt_qn} with Theorem~\ref{thm:biv_FCLT} or Corollary~\ref{cor:biv_FCLT_abs_mixing} respectively, we simply need to show why the causality (condition \eqref{eq:causal}), 
and $X\sim (ARMA-\phi)$ 
or $(ARMA-\beta)$, 
are sufficient for the $L_2$-NED of $(X_0)$ and $\lvert X_t^r \rvert$  with constants $\nu(k)= O(k^{-s}), s>6$, and $\nu(k) = O(k^{-\gamma})$, $\gamma >2$, respectively. 
For $X_0$, this follows directly by \cite[Lemma~3.1]{Qiu11} (by their result, a causal ARMA($p$,$q$) is geometrically $L_2$-NED). For $(\lvert X_0 \rvert^r)$, the geometric $L_2$-NED has been exactly established through Lemma~\ref{lemma:ARMA_NED_Xp}.
Finally, as geometric $L_2$-NED implies a rate of $\nu(k)= O(k^{-s}), s>6$, the necessary rate for the application of Theorem~\ref{thm:biv_FCLT} and Corollary~\ref{cor:biv_FCLT_abs_mixing}, respectively, is attained.
\end{proof}

\subsection{Proofs of Auxiliary Results} \label{ssec:aux_res_proofs}

We start by establishing the asymptotics of the $r$-th absolute centred sample moment (Proposition~\ref{prop:Abs_central_mom_asympt_garch}). To do so, we need the following lemma, which extends Lemma 2.1 in \cite{Segers14} (case $v=1$) to any moment $v \in \N$, as well as the iid case presented in Lemma~A.1 in \cite{Brautigam22}. 
\begin{lemma} \label{lemma:segers_modif_garch}
Consider a stationary and ergodic process $(X_n, n \geq 0)$ with `short-memory', i.e. \\
$\sum_{i=0}^{\infty} \lvert \Cov(X_0, X_i) \rvert < \infty$. 
Then, for $v=1$ or $2$, given that the 2nd moment of $X_0$ exists, or, for any integer $v > 2$, given that the $v$-th moment of $X_0$ exists, it holds that, as $n \rightarrow \infty$, 
\begin{eqnarray} \label{eq:segers_lemma_eq_garch}
&&\frac{1}{n} \sum_{i=1}^n (X_i - \mu)^{v} \left( \lvert X_i - \bar{X}_n \rvert - \lvert X_i - \mu \rvert \right) = \nonumber\\
&&(\bar{X}_n - \mu) \times \E[(X_0 - \mu )^v \sgn(\mu-X_0)] + o_P(1/\sqrt{n}).
\end{eqnarray}
\end{lemma}
\begin{proof}{\bf of Lemma~\ref{lemma:segers_modif_garch} and of Proposition~\ref{prop:Abs_central_mom_asympt_garch}.}
  The proofs of Lemma~\ref{lemma:segers_modif_garch} and Proposition~\ref{prop:Abs_central_mom_asympt_garch} follow the lines of their equivalents in the iid case; see the proof of Lemma~A.1 and Proposition~A.1 in \cite{Brautigam22}.
  In the dependent case, it needs to be adapted, 
  using the stationarity, ergodicity and short-memory of the process. By these three properties, it follows that $\sqrt{n} \lvert \bar{X}_n - \mu \rvert^{v+1} \underset{n \rightarrow \infty}{\overset{P} \rightarrow} 0$ for any integer $v \geq 1$. 
  Further, 
  we use the ergodicity of the process, instead of the strong law of large numbers, to conclude that
  \[ \frac{1}{n} \sum_{i=1}^n (X_i - \mu)^v \sgn(\mu-X_i) \underset{n \rightarrow \infty}{\overset{a.s.}\rightarrow} \E[(X_0 - \mu)^v \sgn(\mu-X_0)]. 
  \]
  Finally, for the proof of Proposition~\ref{prop:Abs_central_mom_asympt_garch},we apply Lemma~\ref{lemma:segers_modif_garch} instead of its counterpart in the iid case (\cite[Lemma~A.1]{Brautigam22}).
  \end{proof}
\begin{proof}{\bf of Lemma~\ref{lemma:L2NED_bd_fct}.}

(i) First, let us observe that
\begin{align*}
 \E[ \lvert g(X_0) - E[g(X_0) \vert \mathcal{F}_{-k}^{+k}] \rvert^2 ] &\leq \E[ \lvert g(X_0) - g(E[X_0 \vert \mathcal{F}_{-k}^{+k}]) \rvert^2 ]
\\ &= \E[ \lvert g(X_0) - g(E[X_0 \vert \mathcal{F}_{-k}^{+k}]) \rvert^2 \1_{ \left( \lvert X_0 - \E[X_0 \vert \mathcal{F}_{-k}^{+k}] \rvert \leq \sqrt{\nu(k)} \right) } ] \numberthis \label{eq:part1}
\\ &\quad + \E[ \lvert g(X_0) - g(E[X_0 \vert \mathcal{F}_{-k}^{+k}]) \rvert^2 \1_{\left( \lvert X_0 - \E[X_0 \vert \mathcal{F}_{-k}^{+k}] \rvert \geq \sqrt{\nu(k)} \right) } ], \numberthis \label{eq:part2}
\end{align*}
where the first inequality follows from the fact that 
$\E[\left(X- \E[X \vert Z]\right)^2] \leq \E[\left( X - h(Z) \right)^2]$ for any measurable function $h$, in particular for $h=g$. 
For \eqref{eq:part1}, we obtain by using the 2-variation condition,
\begin{eqnarray*}
&&\E[ \lvert g(X_0) - g(E[X_0 \vert \mathcal{F}_{-k}^{+k}]) \rvert^2 \1_{ \left(  \lvert X_0 - \E[X_0 \vert \mathcal{F}_{-k}^{+k}] \rvert \leq \sqrt{\nu(k)} \right) } ]\\
 &&\leq \; \E[ \sup_{x \in \R: \lvert x -X_0 \rvert  \leq \sqrt{\nu(k)} } \lvert g(X_0) - g(x)  \rvert^2] 
\;\leq \; L  \nu(k).
\end{eqnarray*}
For \eqref{eq:part2}, using the boundedness of $g$ by $K$, we have
{\small
$$
\E[ \lvert g(X_0) - g(E[X_0 \vert \mathcal{F}_{-k}^{+k}]) \rvert^2 \1_{ \lvert X_0 - \E[X_0 \vert \mathcal{F}_{-k}^{+k}] \rvert \geq \sqrt{\nu(k)}} ] 
\leq 4 K^2 \P \left( \lvert X_0 - \E[X_0 \vert \mathcal{F}_{-k}^{+k}] \rvert \geq \sqrt{\nu(k)} \right).
$$
}
Then, by the extended version of the Markov inequality and the fact that, by assumption, $(X_0)$ is $L_2$-NED with constant $\nu(k)$, we obtain
\[ \P \left( \lvert X_0 - \E[X_0 \vert \mathcal{F}_{-k}^{+k}] \rvert \geq \sqrt{\nu(k)} \right) \leq  \frac{\| X_0 - \E[X_0 \vert \mathcal{F}_{-k}^{+k}] \|_2^2}{\nu(k)} \leq  \nu(k), \]
from which we deduce that \eqref{eq:part2} can be bounded by $4K^2 \nu(k)$. 
Hence, we can conclude that
\[ \E[ \lvert g(X_0) - E[g(X_0) \vert \mathcal{F}_{-k}^{+k}] \rvert^2 ] \leq \eqref{eq:part1} + \eqref{eq:part2} \leq \left( L + 4K^2 \right) \nu(k) \]
\[ \text{i.e.~} \|  \lvert g(X_0) - E[g(X_0) \vert \mathcal{F}_{-k}^{+k}] \|_2 \leq \sqrt{ \left( L + 4K^2 \right) \nu(k) }. \]
%
(ii) We follow the steps of the proof of (i) 
except that for \eqref{eq:part1}, we use the specificities of the indicator function $g(x) := \1_{\left( x \leq t \right)}$ ($t \in \R$) and use the 1-variation condition (instead of the 2-variation condition) to conclude the same bound:
\begin{eqnarray*}
&&\E[ \lvert g(X_0) - g(E[X_0 \vert \mathcal{F}_{-k}^{+k}]) \rvert^2 \1_{ \left(  \lvert X_0 - \E[X_0 \vert \mathcal{F}_{-k}^{+k}] \rvert \leq \sqrt{\nu(k)} \right) } ]\\
 &&\leq \; \E[ \sup_{x \in \R: \lvert x -X_0 \rvert  \leq \sqrt{\nu(k)} } \lvert g(X_0) - g(x)  \rvert^2] 
= \; \E[ \sup_{x \in \R: \lvert x -X_0 \rvert  \leq \sqrt{\nu(k)} } \lvert g(X_0) - g(x)  \rvert ] 
\;\leq \; L  \sqrt{\nu(k)} \;\leq \; L  \nu(k).
\end{eqnarray*}
The next steps remain the same as in (i); we can thus conclude to \eqref{eq:CS-L2Ned}.
\end{proof}
\begin{proof} {\bf of Lemma~\ref{lemma:NED_for_polys_general}.}
We show the $L_2$-NED of $(f(X_n))_{n \in \Z}$ by directly estimating the constants. We start with the expression $\| f(X_n) - \E[ f(X_n) \vert \mathcal{F}_{-k}^{+k}] \|_2^2$ and comment line by line on the inequalities we use. 
\\By the definition of conditional expectation as minimizer in $L_2$ norm, and then the inequality $\left(|a| - |b| \right)^2 \leq \lvert a^2 - b^2 \rvert$, for $ a,b \in \R$, we can write
\begin{eqnarray}
\| f(X_n) - \E[ f(X_n) \vert \mathcal{F}_{-k}^{+k}] \|_2^2 &\leq& \| f( X_n) - f( \E[  X_n \vert \mathcal{F}_{-k}^{+k}])  \|_2^2 = \E[ \lvert f(X_n) - f(\E[  X_n \vert \mathcal{F}_{-k}^{+k}]) \rvert^2] \notag
\\ &\leq & \E \left\lvert f^2(X_n) - f^2(\E[  X_n \vert \mathcal{F}_{-k}^{+k}]) \right\rvert. \label{eq:RHS1} 
\end{eqnarray}
Denoting $g=f^2$ and using the mean value theorem, $g(x)- g(y) = g'(c) (x-y), ~\text{for~} c= ax + (1-a)y \text{~with~} a \in [0,1]$, we can bound the RHS of \eqref{eq:RHS1} as
$$
    \E \left\lvert g(X_n) - g(\E[  X_n \vert \mathcal{F}_{-k}^{+k}]) \right\rvert 
  \leq \E \left\lvert  g' \left( a X_n + (1-a) \E[X_n \vert \mathcal{F}_{-k}^{+k}] \right) \left( X_n - \E[ X_n \vert \mathcal{F}_{-k}^{+k}] \right) \right\rvert 
$$ 
hence, after noticing that $g' = 2f f'$ is an increasing function and using the triangle inequality for the absolute value, it comes
\begin{equation}\label{eq:RHS2}
  \E \left\lvert g(X_n) - g(\E[  X_n \vert \mathcal{F}_{-k}^{+k}]) \right\rvert \leq \E [\lvert  g' \left(  a\lvert X_n \rvert + (1-a) \lvert \E[X_n \vert \mathcal{F}_{-k}^{+k}] \rvert \right) \rvert  ~~ \lvert  X_n - \E[ X_n \vert \mathcal{F}_{-k}^{+k}] \rvert ]. 
\end{equation}
Now, 
applying Jensen's inequality (as $g'$ is convex on $\R^{+}$ by assumption), then, H\"{o}lder's one, and finally the triangle inequality, the RHS of \eqref{eq:RHS2} can be bounded by 
\begin{align}
& \E [ \lvert  a g'(\lvert X_n \rvert) + (1-a) g'(\lvert \E[X_n \vert \mathcal{F}_{-k}^{+k}] \rvert) \rvert ~~ \lvert X_n - \E[ X_n \vert \mathcal{F}_{-k}^{+k}]  \rvert]  \notag
\\ & 
\leq \| a g'(\lvert X_n \rvert) + (1-a) g'(\lvert \E[X_n \vert \mathcal{F}_{-k}^{+k}]\rvert) \|_p  ~~\| X_n - \E[ X_n \vert \mathcal{F}_{-k}^{+k}] \|_q \notag
\\ & 
\leq \left( a \| g'(\lvert X_n \rvert ) \|_p + (1-a) \| g'(\lvert \E[X_n \vert \mathcal{F}_{-k}^{+k}] \rvert) \|_p \right) \| X_n - \E[ X_n \vert \mathcal{F}_{-k}^{+k}] \|_q . \label{eq:RHS3}
\end{align}
For the last step, notice that the composed function $g'(\lvert \cdot \rvert)^p$ is convex: As $(g')^p$ is non-decreasing for $x \geq 0$, 
it holds by the convexity of the absolute value function:
\begin{align*}
\left( g'(\lvert \alpha x + (1-\alpha) y \rvert) \right)^p &\leq \left( g'(\alpha \lvert x \rvert + (1-\alpha) \lvert y \rvert) \right)^p
\leq \alpha \left( g'(\lvert x \rvert) \right)^p + (1-\alpha) \left( g'(\lvert y \rvert) \right)^p,
\end{align*}
where the second inequality follows by the convexity of $(g')^p$ on $\R^{+}$, which also holds by assumption.
Consequently, we can apply Jensen's inequality for conditional expectations to this compound function and obtain an equivalent expression for the RHS of \eqref{eq:RHS3}:
{\small
\begin{align*}
\left( a \| g'(\lvert X_n \rvert) \|_p + (1-a) \| g'(\lvert X_n \rvert) \|_p \right) \| X_n - \E[ X_n \vert \mathcal{F}_{-k}^{+k}] \|_q 
= \| g'(\lvert X_n \rvert) \|_p  \| X_n - \E[ X_n \vert \mathcal{F}_{-k}^{+k}] \|_q  .
\end{align*}
}
Thus, combining all these calculations and inequalities, we can conclude that
\[ \| f(X_n) - \E[ f(X_n) \vert \mathcal{F}_{-k}^{+k}] \|_2 \leq \sqrt{ \| g'(\lvert X_n \rvert) \|_p  \| X_n - \E[ X_n \vert \mathcal{F}_{-k}^{+k}] \|_q } \leq  O(\nu^{1/2} (k)), \]
where the last inequality holds using the assumptions $L_q$-NED of $(X_n)_{n \in \Z}$ and the $L_p$-boundedness of $g'(\lvert X_n \rvert)$.
\end{proof}
\begin{proof}{\bf of Lemma~\ref{cor:NED_for_polys}.}
Consider the function $f$ defined by $f(x) = \lvert x \rvert^r$.
We have $f'(x) = r \lvert x \rvert^{r-1} \sgn(x)$ and $f''(x) = r (r-1) \lvert x \rvert^{r-2}$. Thus, we can verify the conditions of Lemma~\ref{lemma:NED_for_polys_general}, as $f$ is positive and convex ($f'' \geq 0$),
$f\times f'$ is positive for $x\ge 0$ and also convex for $x \ge 0$ (since $f f'(x)= r x^{2r-1}$ and $(f f')'' = r (2r-1) (2r-2) x^{2r-3}$ for  $r \neq 1$ (and 0 for $r=1$)). Finally, $(f f')^p$ is convex for $x \geq 0$, for a choice of $p$. Indeed, as $(f f')^p = r x^{p (2r-1)}$, we get $((f f')^p)'' = r p^2 (2r-1) (2r-2) x^{p (2r-1) -2}$ for $r \neq \frac{p+1}{2p}$ (and equal to $0$ for $r= \frac{p+1}{2p}$) -   which is, for all choices of $p \in (1, \infty)$, positive for $x\geq 0$. 
\end{proof}
\begin{proof}{\bf of Lemma~\ref{lemma:ARMA_NED_Xp}.}
The proof consists of three steps.
First, the process $X$ being a causal ARMA($p$, $q$)-process, we can apply Theorem 3.1 from \cite{Qiu11}. Choosing $p=2$ therein, we can conclude that $X$ is strong $L_2$-NED with rate $\nu(k)= \rho^k$, for some $0< \rho <1$, hence it is $L_2$-NED with the same rate (compare Definitions~1.1 and 1.2 in \cite{Qiu11}).
Second, since $X$, equivalently $(X_0)$, is $L_2$-NED, 
we can apply Lemma~\ref{cor:NED_for_polys}. As it holds by assumption that $\lvert X_0 \rvert^{2r}< \infty$, we choose $q=2r$ and accordingly $p= \frac{2r}{2r-1}$. 
Then, by Lemma~\ref{cor:NED_for_polys}, it holds that $(\lvert X_0 \rvert^r)$ is geometrically $L_2$-NED if $(X_0)$ is geometrically $L_{2r}$-NED.
Finally, let us prove that this sufficient condition holds.  
For this, recall the representation of an ARMA process (which holds for an ARMA($p$, $q$) process satisfying \eqref{eq:ARMA_pq}, \eqref{eq:causal}, see Lemma 3.1 in \cite{Qiu11}):
\begin{equation}\label{eq:Qiu11}
X_t = \sum_{j=0}^{\infty} \psi_j \epsilon_{t-j}, \text{~with~} \lvert \psi_j \rvert = O(\rho^j) \text{~for some~} 0< \rho <1.
\end{equation}
Then, we apply a standard truncation argument (as e.g. done in \cite{Lee14}, Lemma 1 for augmented GARCH($p$, $q$) processes):
Define the truncated variable $h_{t,m} = \sum_{j=0}^m \psi_j \epsilon_{t-j}$, which, by construction, is $\mathcal{F}_{\epsilon,t-m,t+m}$-measurable. Let $h_{t,m}^{*} := X_t - h_{t,m}$.
For any given integer $r>0$, let us now verify the $L_{2r}$-NED of $(X_0)$: 
$$
\| X_0 - \E[ X_0 \vert \mathcal{F}_{\epsilon, -m, +m}] \|_{2r}^{2r}  = \| h_{0,m}^{*} - \E[ h_{0,m}^{*} \vert \mathcal{F}_{\epsilon, -m, +m}] \|_{2r}^{2r} \leq 2^{1+2r} \|  h_{0,m}^{*} \|_{2r}^{2r},
$$
where the first equality follows from the $\mathcal{F}_{\epsilon,-m,+m}$-measurability of $h_{0,m}$ and the subsequent inequality because $\E[\lvert X - \E[X \vert Y] \rvert^{2r}] \leq 2^{(1+2r)} \E[\lvert X \rvert^{2r}]$ 
for any random variables $X,Y$ and $r\geq 1$. 

Thus, it is enough to prove the geometric $L_{2r}$-NED of $ h_{0,m}^{*}$ to conclude the proof. Since we can write
$$
\| h_{0,m}^{*} \|_{2r}  = \| \sum_{j=m+1}^{\infty}  \psi_j \epsilon_{-j}  \|_{2r}
 \leq \sum_{j=m+1}^{\infty}  \|  \psi_j \epsilon_{-j}  \|_{2r} 
= \sum_{j=m+1}^{\infty}  \lvert \psi_j \rvert \|  \epsilon_{-j}  \|_{2r} 
= \|  \epsilon_{0}  \|_{2r}  \sum_{j=m+1}^{\infty}  \lvert \psi_j \rvert,
$$
where $\epsilon_0$ has a finite second moment by definition of the ARMA($p$, $q$) process, then, using \eqref{eq:Qiu11}, we obtain that,
for a constant $C>0$,
$$
\| h_{0,m}^{*} \|_{2r}  \leq \|  \epsilon_{0}  \|_{2r}  \sum_{j=m+1}^{\infty}  \lvert \psi_j \rvert  \leq C \sum_{j=m+1}^{\infty} \rho^j 
= C \rho^{m+1} \sum_{j=0}^{\infty} \rho^j = C \frac{\rho^{m+1}}{1-\rho} 
= O(e^{-m \, \eta}),
$$
for $\eta = -\log{(\rho)}$, hence the result. 
\end{proof}

\section*{Financial disclosure}

None reported.

\section*{Conflict of interest}

The authors declare no potential conflict of interests.


\small
\bibliography{art-mixing_dependence-2024nov}

\begin{thebibliography}{10}

\bibitem{Anderson71}
T.~W. Anderson.
\newblock {\em The statistical analysis of time series}.
\newblock John Wiley \& Sons, 1971.

\bibitem{Andrews88}
D.~Andrews.
\newblock Laws of large numbers for dependent non-identically distributed
  random variables.
\newblock {\em Econometric Theory}, 4(3):458--467, 1988.

\bibitem{Aue06}
A.~Aue, I.~Berkes, and L.~Horv{\'a}th.
\newblock Strong approximation for the sums of squares of augmented {GARCH}
  sequences.
\newblock {\em Bernoulli}, 12(4):583--608, 2006.

\bibitem{Bahadur66}
R.~Bahadur.
\newblock A note on quantiles in large samples.
\newblock {\em The Annals of Mathematical Statistics}, 37(3):577--580, 1966.

\bibitem{Berkes08}
I.~Berkes, S.~H{\"o}rmann, and L.~Horv{\'a}th.
\newblock The functional central limit theorem for a family of {GARCH}
  observations with applications.
\newblock {\em Statistics \& Probability Letters}, 78(16):2725--2730, 2008.

\bibitem{Billingsley68}
P.~Billingsley.
\newblock {\em Convergence of probability measures}.
\newblock John Wiley \& Sons, 1st edition, 1968.

\bibitem{Bollerslev86}
T.~Bollerslev.
\newblock Generalized autoregressive conditional heteroskedasticity.
\newblock {\em Journal of Econometrics}, 31(3):307--327, 1986.

\bibitem{Bollerslev08}
T.~Bollerslev.
\newblock Glossary to {ARCH} ({GARCH}).
\newblock {\em CREATES Research Paper}, 49, 2008.

\bibitem{Borovkova01}
S.~Borovkova, R.~Burton, and H.~Dehling.
\newblock Limit theorems for functionals of mixing processes with applications
  to u-statistics and dimension estimation.
\newblock {\em Transactions of the American Mathematical Society},
  353(11):4261--4318, 2001.

\bibitem{Bougerol92}
P.~Bougerol and N.~Picard.
\newblock Stationarity of {GARCH} processes and of some nonnegative time
  series.
\newblock {\em Journal of econometrics}, 52(1-2):115--127, 1992.

\bibitem{Boussama98}
F.~Boussama.
\newblock {\em Ergodicit{\'e}, m{\'e}lange et estimation dans les modeles
  GARCH}.
\newblock PhD thesis, Universit\'e Paris 7, 1998.

\bibitem{Box15}
G.~E. Box, G.~M. Jenkins, G.~C. Reinsel, and G.~M. Ljung.
\newblock {\em Time series analysis: forecasting and control}.
\newblock John Wiley \& Sons, 2015.

\bibitem{Bradley05}
R.~Bradley.
\newblock Basic properties of strong mixing conditions. {A} survey and some
  open questions.
\newblock {\em Probability surveys}, 2:107--144, 2005.

\bibitem{Brautigam22}
M.~Br{\"a}utigam, M.~Dacorogna, and M.~Kratz.
\newblock Pro-cyclicality beyond business cycles.
\newblock {\em Mathematical Finance}, 33(2):308--341, 2023.

\bibitem{Brockwell91}
P.~J. Brockwell and R.~A. Davis.
\newblock {\em Time series: theory and methods}.
\newblock Springer Science \& Business Media, 1991.

\bibitem{Carrasco02}
M.~Carrasco and X.~Chen.
\newblock Mixing and moment properties of various {GARCH} and stochastic
  volatility models.
\newblock {\em Econometric Theory}, 18(1):17--39, 2002.

\bibitem{Davidson02}
J.~Davidson.
\newblock Establishing conditions for the functional central limit theorem in
  nonlinear and semiparametric time series processes.
\newblock {\em Journal of Econometrics}, 106(2):243--269, 2002.

\bibitem{Davis98}
R.~Davis and T.~Mikosch.
\newblock The sample autocorrelations of heavy-tailed processes with
  applications to {ARCH}.
\newblock {\em The Annals of Statistics}, 26(5):2049--2080, 1998.

\bibitem{Davis99}
R.~A. Davis, T.~Mikosch, and B.~Basrak.
\newblock Sample {ACF} of multivariate stochastic recurrence equations with
  application to {GARCH}.
\newblock {\em Preprint, available at www. math. ku. dk/mikosch}, 1999.

\bibitem{DeJong00}
R.~M. De~Jong and J.~Davidson.
\newblock The functional central limit theorem and weak convergence to
  stochastic integrals {I}: weakly dependent processes.
\newblock {\em Econometric Theory}, 16(5):621--642, 2000.

\bibitem{Denker86}
M.~Denker and G.~Keller.
\newblock Rigorous statistical procedures for data from dynamical systems.
\newblock {\em Journal of Statistical Physics}, 44(1-2):67--93, 1986.

\bibitem{Ding93}
Z.~Ding, C.~Granger, and R.~Engle.
\newblock A long memory property of stock market returns and a new model.
\newblock {\em Journal of Empirical Finance}, 1(1):83--106, 1993.

\bibitem{Duan97}
J.~Duan.
\newblock Augmented {GARCH} (p, q) process and its diffusion limit.
\newblock {\em Journal of Econometrics}, 79(1):97--127, 1997.

\bibitem{Engle82}
R.~Engle.
\newblock Autoregressive conditional heteroscedasticity with estimates of the
  variance of united kingdom inflation.
\newblock {\em Econometrica: Journal of the Econometric Society},
  50(4):987--1007, 1982.

\bibitem{Engle93}
R.~Engle and V.~Ng.
\newblock Measuring and testing the impact of news on volatility.
\newblock {\em The Journal of Finance}, 48(5):1749--1778, 1993.

\bibitem{Geweke86}
J~Geweke.
\newblock Modeling the persistence of conditional variances: a comment.
\newblock {\em Econometric Reviews}, 5:57--61, 1986.

\bibitem{Giot04}
P.~Giot and S.~Laurent.
\newblock Modelling daily value-at-risk using realized volatility and {ARCH}
  type models.
\newblock {\em Journal of empirical finance}, 11(3):379--398, 2004.

\bibitem{Giraitis06}
L.~Giraitis, R.~Leipus, and D.~Surgailis.
\newblock Recent advances in {ARCH} modelling.
\newblock In G.~Teyssière and A.P. Kirman, editors, {\em Long Memory in
  Economics}, pages 3--38. Springer, 2007.

\bibitem{Glosten93}
L.~Glosten, R.~Jagannathan, and D.~Runkle.
\newblock On the relation between the expected value and the volatility of the
  nominal excess return on stocks.
\newblock {\em The Journal of Finance}, 48(5):1779--1801, 1993.

\bibitem{Hesse90}
CH. Hesse.
\newblock A {B}ahadur-type representation for empirical quantiles of a large
  class of stationary, possibly infinite-variance, linear processes.
\newblock {\em The Annals of Statistics}, pages 1188--1202, 1990.

\bibitem{Higgins92}
M.~Higgins and A.~Bera.
\newblock A class of nonlinear {ARCH} models.
\newblock {\em International Economic Review}, 33(1):137--158, 1992.

\bibitem{Hoga19}
Y.~Hoga.
\newblock Confidence intervals for conditional tail risk measures in
  arma--garch models.
\newblock {\em Journal of Business \& Economic Statistics}, 37(4):613--624,
  2019.

\bibitem{Holan10}
S.~H. Holan, R.~Lund, and G.~Davis.
\newblock The {ARMA} alphabet soup: A tour of {ARMA} model variants.
\newblock {\em Statistics Surveys}, 4:232--274, 2010.

\bibitem{Hormann08}
S.~H{\"o}rmann.
\newblock Augmented {GARCH} sequences: Dependence structure and asymptotics.
\newblock {\em Bernoulli}, 14(2):543--561, 2008.

\bibitem{Horvath08}
L.~Horv{\'a}th and P.~Kokoszka.
\newblock Sample autocovariances of long-memory time series.
\newblock {\em Bernoulli}, 14(2):405--418, 2008.

\bibitem{Ibragimov71}
I.~Ibragimov and Y.~Linnik.
\newblock {\em Independent and stationary sequences of random variables}.
\newblock Wolters, Noordhoff Pub., 1971.

\bibitem{Ibragimov62}
I.~A. Ibragimov.
\newblock Some limit theorems for stationary processes.
\newblock {\em Theory of Probability \& Its Applications}, 7(4):349--382, 1962.

\bibitem{Jeon13}
J.~Jeon and J.~W. Taylor.
\newblock Using {CAViaR} models with implied volatility for value-at-risk
  estimation.
\newblock {\em Journal of Forecasting}, 32(1):62--74, 2013.

\bibitem{Kulik12}
R.~Kulik and P.~Soulier.
\newblock Limit theorems for long-memory stochastic volatility models with
  infinite variance: partial sums and sample covariances.
\newblock {\em Advances in Applied Probability}, 44(4):1113--1141, 2012.

\bibitem{Lee14}
O.~Lee.
\newblock Functional central limit theorems for augmented {GARCH} (p, q) and
  {FIGARCH} processes.
\newblock {\em Journal of the Korean Statistical Society}, 43(3):393--401,
  2014.

\bibitem{Lindner09}
A.~Lindner.
\newblock Stationarity, mixing, distributional properties and moments of
  {GARCH} (p, q)--processes.
\newblock In T.~Mikosch, JP. Kreiß, R.~Davis, and T.~Andersen, editors, {\em
  Handbook of financial time series}, pages 43--69. Springer, 2009.

\bibitem{Ling97}
S.~Ling and W.K. Li.
\newblock On fractionally integrated autoregressive moving-average time series
  models with conditional heteroscedasticity.
\newblock {\em Journal of the American Statistical Association},
  92(439):1184--1194, 1997.

\bibitem{Ling98}
S.~Ling and W.K. Li.
\newblock Limiting distributions of maximum likelihood estimators for unstable
  autoregressive moving-average time series with general autoregressive
  heteroscedastic errors.
\newblock {\em Annals of Statistics}, 26(1):84--125, 1998.

\bibitem{Ling02}
S.~Ling and M.~McAleer.
\newblock Necessary and sufficient moment conditions for the {GARCH} (r, s) and
  asymmetric power {GARCH} (r, s) models.
\newblock {\em Econometric theory}, pages 722--729, 2002.

\bibitem{Ling03}
S.~Ling and M.~McAleer.
\newblock Asymptotic theory for a vector {ARMA-GARCH} model.
\newblock {\em Econometric theory}, pages 280--310, 2003.

\bibitem{Milhoj87}
A.~Milh{\o}j.
\newblock A multiplicative parameterization of {ARCH} models.
\newblock {\em Working Paper}, 1987.

\bibitem{Nelson91}
D.~Nelson.
\newblock Conditional heteroskedasticity in asset returns: A new approach.
\newblock {\em Econometrica: Journal of the Econometric Society},
  59(2):347--370, 1991.

\bibitem{Pantula86}
S.~Pantula.
\newblock Modeling the persistence of conditional variances: {A} comment.
\newblock {\em Econometric Reviews}, 5:79--97, 1986.

\bibitem{Qiu11}
J.~Qiu and Z.~Lin.
\newblock The functional central limit theorem for linear processes with strong
  near-epoch dependent innovations.
\newblock {\em Journal of mathematical analysis and applications},
  376(1):373--382, 2011.

\bibitem{Schwert89}
G.~Schwert.
\newblock Why does stock market volatility change over time?
\newblock {\em The Journal of Finance}, 44(5):1115--1153, 1989.

\bibitem{Segers14}
J.~Segers.
\newblock On the asymptotic distribution of the mean absolute deviation about
  the mean.
\newblock {\em arXiv:1406.4151}, 2014.

\bibitem{Sen68}
P.~K. Sen.
\newblock Asymptotic normality of sample quantiles for m-dependent processes.
\newblock {\em The annals of mathematical statistics}, pages 1724--1730, 1968.

\bibitem{Sen72}
P.~K. Sen.
\newblock On the {B}ahadur representation of sample quantiles for sequences of
  $\varphi$-mixing random variables.
\newblock {\em Journal of Multivariate analysis}, 2(1):77--95, 1972.

\bibitem{Song18}
J.~Song and J.~Kang.
\newblock Parameter change tests for arma--garch models.
\newblock {\em Computational Statistics \& Data Analysis}, 121:41--56, 2018.

\bibitem{Spierdijk16}
L.~Spierdijk.
\newblock Confidence intervals for arma--garch value-at-risk: The case of heavy
  tails and skewness.
\newblock {\em Computational statistics \& data analysis}, 100:545--559, 2016.

\bibitem{Taylor86}
S.~Taylor.
\newblock {\em Modelling financial time series}.
\newblock Wiley, New York, 1986.

\bibitem{VanDerVaart98}
A.~Van~der Vaart.
\newblock {\em Asymptotic statistics}.
\newblock Cambridge University Press, 1998.

\bibitem{Weiss84}
A.~A. Weiss.
\newblock {ARMA} models with {ARCH} errors.
\newblock {\em Journal of time series analysis}, 5(2):129--143, 1984.

\bibitem{Weiss86}
A.A. Weiss.
\newblock Asymptotic theory for {ARCH} models: estimation and testing.
\newblock {\em Econometric theory}, pages 107--131, 1986.

\bibitem{Wendler11}
M.~Wendler.
\newblock Bahadur representation for u-quantiles of dependent data.
\newblock {\em Journal of Multivariate Analysis}, 102(6):1064--1079, 2011.

\bibitem{Wu05}
W.~B. Wu.
\newblock On the {B}ahadur representation of sample quantiles for dependent
  sequences.
\newblock {\em The Annals of Statistics}, 33(4):1934--1963, 2005.

\bibitem{Wu21}
Y.~Wu, W.~Yu, and X.~Wang.
\newblock The bahadur representation of sample quantiles for ?-mixing random
  variables and its application.
\newblock {\em Statistics}, 55(2):426--444, 2021.

\bibitem{Zakoian94}
J.-M. Zakoian.
\newblock Threshold heteroskedastic models.
\newblock {\em Journal of Economic Dynamics and Control}, 18(5):931--955, 1994.

\bibitem{Zumbach12}
G.~Zumbach.
\newblock {\em Discrete Time Series, Processes, and Applications in Finance}.
\newblock Springer Science \& Business Media, 2012.

\end{thebibliography}
\bibliographystyle{acm}

\begin{appendix}
  \section*{Examples of Augmented GARCH models}\label{appn} 

First, let us give a schematic overview of the nesting of some augmented GARCH($p$,$q$) models in Figure~\ref{fig:GARCH_nested}, with a brief description of the acronyms, authors, and relations between these processes. Then, we explicitly state the conditions on the moments and parameters of these processes, for the bivariate asymptotics of Corollary~\ref{thm:augm_GARCH_pq_asympt_qn} to be valid, in view of applications.
\begin{figure}[h]
  %
  \centering
  \includegraphics[scale=0.5]{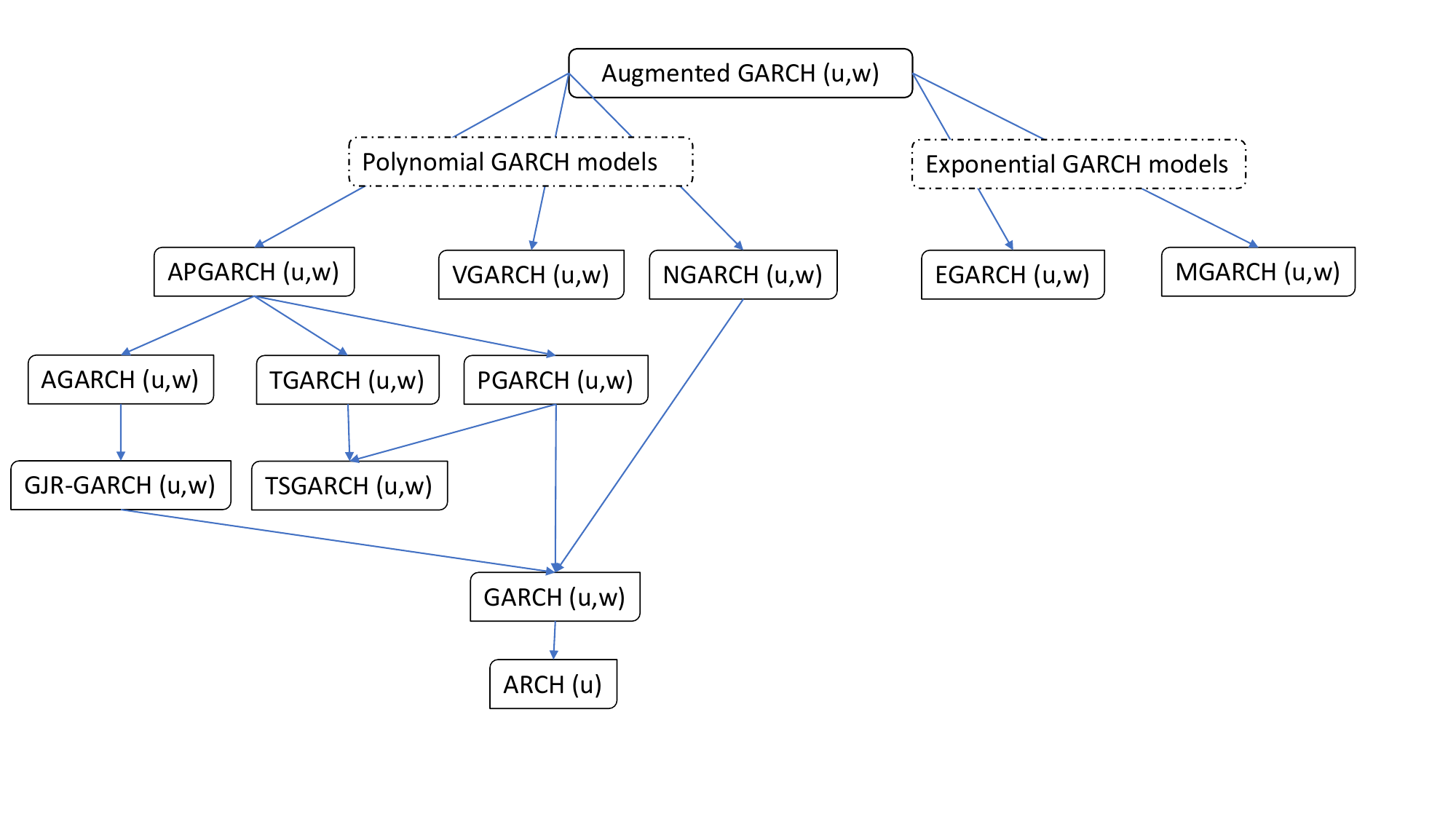}
  \vspace{-10ex}
  \caption{\label{fig:GARCH_nested}\sf\small Schematic overview of the nesting of some augmented GARCH($p$,$q$) models.}
  \end{figure} 

The restrictions on the parameters, if not specified differently, are $\omega \geq0, \alpha_i \geq0, -1 \leq\gamma_i \leq 1, \beta_j \geq 0$ for $i=1,...,p$, $j=1,...,q$.
\begin{itemize}
\item APGARCH: Asymmetric power GARCH, introduced by Ding et al. in \cite{Ding93}. One of the most general polynomial GARCH models.
\item AGARCH: Asymmetric GARCH, defined also by Ding et al. in \cite{Ding93}, choosing $\delta =1$ in APGARCH.
\item GJR-GARCH: This process is named after its three authors Glosten, Jaganathan and Runkle and was defined by them in \cite{Glosten93}. For the parameters $\alpha_i^{\ast} , \gamma_i^{\ast}$ it holds that $\alpha_i^{\ast} = \alpha_i (1-\gamma_i)^2$ and $\gamma_i^{\ast}= 4 \alpha_i \gamma_i$.
\item GARCH: Choosing all $\gamma_i=0$ in the AGARCH model (or $\gamma_i^{\ast} = 0$ in the GJR-GARCH), gives back the well-known GARCH($p$,$q$) process by Bollerslev in \cite{Bollerslev86}.
\item ARCH: Introduced by Engle in \cite{Engle82}. We recover it by setting all $\gamma_i =\beta_j=0, \forall i,j$.
\item TGARCH: Choosing $\delta=1/2$ in the APGARCH model leads us the so called threshold GARCH (TGARCH) by Zakoian in \cite{Zakoian94}. For the parameters $\alpha_i^{+} , \alpha_i^{-}$ it holds that $\alpha_i^{+} = \alpha_i (1-\gamma_i), \alpha_i^{-}= \alpha_i (1+ \gamma_i)$.
\item TSGARCH: Choosing $\gamma_i=0$ in the TGARCH model we get, as a subcase, the TSGARCH model, named after its authors, i.e. Taylor, \cite{Taylor86}, and Schwert, \cite{Schwert89}. 
\item PGARCH: Another subfamily of the APGARCH processes is the Power-GARCH (PGARCH), also called sometimes NGARCH (i.e. non-linear GARCH) due to Higgins and Bera in \cite{Higgins92}.
\item VGARCH: The volatility GARCH (VGARCH) model by Engle and Ng in \cite{Engle93} is also a polynomial GARCH model but is not part of the APGARCH family.
\item NGARCH: This non-linear asymmetric model is due to Engle and Ng in \cite{Engle93},  
and sometimes also called NAGARCH. 
\item MGARCH: This model is called multiplicative or logarithmic GARCH and goes back to independent suggestions, in slightly different formulations, of Geweke in
\cite{Geweke86},  Pantula in \cite{Pantula86} and Milh{\o}j in \cite{Milhoj87}.
\item EGARCH: This model is called exponential GARCH, introduced by Nelson in \cite{Nelson91}.
\end{itemize}
Now, recalling the specific conditions mentioned in Corollary~\ref{thm:augm_GARCH_pq_asympt_qn}, note that the continuity and differentiability conditions, $(C_2^{~'})$, $(C_1^{+})$, each at $q_X(p)$, and $(C_0)$ at $0$ for $r=1$, remain the same for the whole class of augmented GARCH processes. Contrary to that, the moment condition $(M_r)$ imposes different restrictions on the parameters of the underlying process, depending on the given augmented GARCH($p$, $q$) process. To our knowledge, there exists no general result for the class of augmented GARCH($p$, $q$) describing the necessary conditions of the process for a given moment to exist: E.g. in \cite{Hormann08} the case of augmented GARCH($1$,$1$) processes is considered, see Corollary~1 and Proposition~1 therein, and in \cite{Ling02} the GARCH($p,q$) and APGARCH($p$,$q$), see Theorem~2.1 and Theorem~3.2 therein. Sufficient but not necessary conditions, which are easier to verify in practice, follow from results in \cite{Giraitis06} as mentioned in \cite{Lindner09}, Proposition~2.

We focus here on solely explaining how, depending on the specifications~\eqref{eq:augm_GARCH_pq} of the process, the conditions, $(P_{\max(1,r/\delta)})$ for polynomial GARCH or $(L_r)$ for exponential GARCH respectively, translate differently in the various examples.
\begin{table}[h]
  {\scriptsize
  \parbox{350pt}{\caption{\label{tbl:vol_formula_models} \sf\small  Presentation of the volatility equation \eqref{eq:augm_GARCH_pq} and the corresponding specifications of functions $g_i,c_j$ for selected augmented GARCH models.}}
  \vspace{-3ex}
  \begin{center}
  \hspace*{-2.0cm}
  \begin{tabular*}{500pt}{l | lll }
  \hline
  \\[-1.5ex]
   & standard formula for $\Lambda(\sigma_t^2)$ & corresponding specifications of $g_i,c_j$ in~\eqref{eq:augm_GARCH_pq} \\
  \hline\hline
  \\[-1.5ex]
  \\[-2ex] \textbf{Polynomial GARCH}& & & 
  \\[1.5ex]  ~~APGARCH family & $\displaystyle \sigma_t^{2 \delta} = \omega + \sum_{i=1}^p \alpha_i \left( \lvert y_{t-i} \rvert - \gamma_i y_{t-i} \right)^{2 \delta} + \sum_{j=1}^q \beta_j \sigma_{t-j}^{2 \delta}$ & \hspace*{-1.1cm}$g_i = \omega/p \;\text{and}\; c_j =  \alpha_j \left( \lvert \epsilon_{t-j} \rvert - \gamma_j \epsilon_{t-j} \right)^{2 \delta} + \beta_j $ 
  \\[1.5ex] \quad\quad AGARCH & $\displaystyle\sigma_t^{2} = \omega + \sum_{i=1}^p \alpha_i \left( \lvert y_{t-i} \rvert - \gamma_i y_{t-i} \right)^{2 } + \sum_{j=1}^q \beta_j \sigma_{t-j}^{2}$ & \quad \quad ~$c_j =   \alpha_j \left( \lvert \epsilon_{t-j} \rvert - \gamma_j \epsilon_{t-j} \right)^{2} + \beta_j $
  \\[1.5ex] \quad\quad GJR-GARCH & $\displaystyle\sigma_t^2 = \omega +  \sum_{i=1}^p \left( \alpha_i^{\ast} +\gamma_i^{\ast} \1_{(y_{t-i}<0)} \right)y_{t-i}^2 + \sum_{j=1}^q \beta_j \sigma_{t-j}^2 $ & \quad \quad ~$c_j = \beta_j + \alpha_j^{\ast} \epsilon_{t-j}^2 +\gamma_j^{\ast} \max(0,-\epsilon_{t-j})^2   $
  \\[1.5ex] \quad\quad GARCH & $\displaystyle\sigma_t^2 = \omega + \sum_{i=1}^p \alpha_i y_{t-i}^2 + \sum_{j=1}^q \beta_j \sigma_{t-j}^2$ & \quad \quad ~$c_j = \alpha_j \epsilon_{t-j}^2 +\beta_j $
  \\[1.5ex] \quad\quad ARCH & $\displaystyle\sigma_t^2 = \omega + \sum_{i=1}^p \alpha_i y_{t-i}^2$ & \quad \quad ~$c_j =  \alpha_j \epsilon_{t-j}^2 $
  \\[1.5ex] \quad\quad TGARCH & $\displaystyle\sigma_t = \omega + \sum_{i=1}^p \left( \alpha_i^{+} max(y_{t-i},0)  - \alpha_i^{-} min(y_{t-i},0) \right) + \sum_{j=1}^q \beta_j \sigma_{t-j}$ & \quad \quad ~$ c_j = \alpha_j \lvert \epsilon_{t-j} \rvert  - \alpha_j \gamma_j \epsilon_{t-j} + \beta_j  $
  \\[1.5ex] \quad\quad TSGARCH & $\displaystyle\sigma_t = \omega + \sum_{i=1}^p  \alpha_i \lvert y_{t-i} \rvert  + \sum_{j=1}^q \beta_j \sigma_{t-j}$ & \quad \quad ~$ c_j = \alpha_j \lvert \epsilon_{t-j} \rvert  + \beta_j $
  \\[1.5ex] \quad\quad PGARCH & $\displaystyle\sigma_t^{\delta} = \omega + \sum_{i=1}^p \alpha_i \lvert y_{t-i} \rvert^{\delta}  + \sum_{j=1}^q \beta_j  \sigma_{t-j}^{\delta}.$ & \quad \quad ~$c_j =  \alpha_j \lvert \epsilon_{t-j} \rvert^{\delta}  + \beta_j $
  \\[1.5ex] ~~VGARCH & $\displaystyle\sigma_t^2 = \omega + \sum_{i=1}^p \alpha_i (\epsilon_{t-i} + \gamma_i)^2 + \sum_{j=1}^q \beta_j \sigma_{t-j}^2.$ & $g_i = \omega/p + \alpha_i (\epsilon_{t-i} + \gamma_i)^2$ \; and \;$c_j = \beta_j$
  \\[1.5ex] ~~NGARCH & $\displaystyle\sigma_t^2 = \omega + \sum_{i=1}^p \alpha_i (y_{t-i} + \gamma_i \sigma_{t-i})^2  + \sum_{j=1}^q \beta_j \sigma_{t-j}^2 $ & $g_i = \omega/p$\; and\; $c_j =  \alpha_j (\epsilon_{t-j} + \gamma_j)^2 +  \beta_j $
  \\\hline
  \\[-1.5ex] \textbf{Exponential GARCH} & & $c_j = \beta_j$\; and
  \\[1.5ex] ~~MGARCH & $\displaystyle\log(\sigma_t^2) = \omega + \sum_{i=1}^p \alpha_i \log(\epsilon_{t-i}^2) + \sum_{j=1}^q \beta_j \log(\sigma_{t-j}^2)$ & ~~~~$g_i = \omega/p + \alpha_i \log(\epsilon_{t-i}^2)$
  \\[1.5ex] ~~EGARCH & $\displaystyle\log(\sigma_t^2) = \omega + \sum_{i=1}^p \alpha_i \left(  \lvert \epsilon_{t-i} \rvert - \E\lvert \epsilon_{t-i}\rvert \right) + \gamma_i \epsilon_{t-i}  + \sum_{j=1}^q \beta_j \log(\sigma_{t-j}^2)$ & ~~~~$g_i = \omega/p + \alpha_i (\lvert\epsilon_{t-i}\rvert - \E\lvert \epsilon_{t-i} \rvert) + \gamma_i \epsilon_{t-i}$
  \\  [1ex]
  \hline
  \end{tabular*}
  \end{center}
  }
  \end{table}

For this, we introduce in Table~\ref{tbl:vol_formula_models} a non-exhaustive selection of different augmented GARCH($p$,$q$) models, providing for each the corresponding volatility equation, \eqref{eq:augm_GARCH_pq}, and the specifications of the functions $g_i$ and $c_j$. 
We consider 10 models that belong to the group of polynomial GARCH ($\Lambda(x) = x^{\delta}$) and two examples of exponential GARCH ($\Lambda(x) = \log(x)$).  
Note that in Table~\ref{tbl:vol_formula_models} the specification of $g_i$ is the same for the whole APGARCH family (only the $c_j$ change), whereas for the two exponential GARCH models, it is the reverse.
The general restrictions on the parameters are as follows: $\omega>0, \alpha_i\geq 0 , -1 \leq \gamma_i \leq 1, \beta_j \geq 0$ for $i=1,...,p$, $j=1,...,q$.
Further, the parameters in the GJR-GARCH (TGARCH) are denoted with an asterix (with a plus or minus) as they are not the same as in the other models.

In Tables~\ref{tbl:cond_bivariate-conv-11} and~\ref{tbl:cond_bivariate-conv-pq}, we present how the conditions $(P_{\max(1,r/\delta)})$ or $(L_r)$ translate for each model.
Table~\ref{tbl:cond_bivariate-conv-11} treats the specific case of an augmented GARCH($p$,$q$) process with $p=q=1$, whereas Table~\ref{tbl:cond_bivariate-conv-pq} treats the general case for arbitrary $p\geq1, q \geq0$.
\vspace{-1ex}
\begin{table}[h]
{\scriptsize
\parbox{430pt}{\caption{\label{tbl:cond_bivariate-conv-11} \sf\footnotesize Conditions $(P_{\max{(1, r/\delta)}})$ or $(L_r)$ respectively translated for different augmented GARCH($1$,$1$) models. Left column for the general $r$-th absolute centred sample moment, middle for the MAD ($r=1$) and right for the variance ($r=2$).}}
\vspace{-4ex}
\begin{center}
\hspace*{-2cm}
\begin{tabular}{l | lll }
\hline
\\[-1.5ex]
augmented \\ GARCH ($1$, $1$) & $r \in \N$ & $r=1$ & $r=2$ \\
\hline\hline
\\[-1.5ex]
\\[-2ex] APGARCH &  $ \E[\lvert \alpha_1 \left( \lvert \epsilon_0\rvert - \gamma_1 \epsilon_{t-1}\right)^{2 \delta} + \beta_1 \rvert^r] <1$ & $\alpha_1 \E\left[\left( \lvert \epsilon_0\rvert - \gamma_1 \epsilon_{t-1} \right)^{2\delta } \right]+ \beta_1  <1$ & $ \E[\lvert \alpha_1 \left(  \lvert \epsilon_0\rvert - \gamma_1 \epsilon_{t-1} \right)^{2 \delta} + \beta_1 \rvert^2] <1$
\\[1.5ex] AGARCH & $ \E[\lvert \alpha_1 \left( \lvert \epsilon_0\rvert - \gamma_1 \epsilon_{t-1}\right)^{2 } + \beta_1 \rvert^r] <1$ & $\alpha_1 \E\left[\left( \lvert \epsilon_0\rvert - \gamma_1 \epsilon_{t-1} \right)^{2 } \right]+ \beta_1  <1$ & $ \E[\lvert \alpha_1 \left( \lvert \epsilon_0\rvert - \gamma_1 \epsilon_{t-1}\right)^{2} + \beta_1 \rvert^2]<1$
\\[1.5ex] GJR-GARCH & $\E[\lvert \alpha_1^{\ast} \epsilon_0^2 + \beta_1 + \gamma_1^{\ast} \max(0,-\epsilon_0^2) \rvert^r] <1$ & $\alpha_1^{\ast} + \beta_1 + \gamma_1^{\ast} \E[\max(0,-\epsilon_0)^2] <1$ &  $\E[\lvert \alpha_1^{\ast} \epsilon_0^2 + \beta_1 + \gamma_1^{\ast} \max(0,-\epsilon_0^2) \rvert^2] <1$
\\[1.5ex] GARCH &  $\E[(\alpha_1 \epsilon_0^2 + \beta_1)^r]< 1 $ & $ \alpha_1 + \beta_1 < 1$ & $\alpha_1^2 \E[\epsilon_0^4] + \alpha_1 \beta_1 + \beta_1^2 < 1 $ 
\\[1.5ex] ARCH &  $\alpha_1^r \E[\epsilon_0^{2r}]< 1 $ & $ \alpha_1 < 1$ & $\alpha_1^2 \E[\epsilon_0^4] < 1 $ 
\\[1.5ex] TGARCH &  $\E[\lvert \alpha_1 \lvert \epsilon_{t-1} \rvert - \alpha_1 \gamma_1 \epsilon_{t-1} + \beta_1 \rvert^r] <1 $&  $\alpha_1 \E\lvert \epsilon_{t-1} \rvert + \beta_1 < 1$ &  $\E[\lvert \alpha_1 \lvert \epsilon_{t-1} \rvert - \alpha_1 \gamma_1 \epsilon_{t-1} + \beta_1 \rvert^2] <1$
\\[1.5ex] TSGARCH &  $\E[\lvert \alpha_1 \lvert \epsilon_{t-1} \rvert + \beta_1 \rvert^r] <1 $&  $\alpha_1 \E\lvert \epsilon_{t-1} \rvert + \beta_1 < 1$ &  $\E[\lvert \alpha_1 \lvert \epsilon_{t-1} \rvert + \beta_1 \rvert^2] <1$  
\\[1.5ex] PGARCH & $\E[\lvert \alpha_1 \lvert \epsilon_0 \rvert +\beta_1\rvert^{2r}] < 1$ & $\alpha_1  + 2 \alpha_1 \beta_1 \E\lvert \epsilon_0 \rvert +\beta_1^2 < 1$ & $ \E[\lvert \alpha_1 \lvert \epsilon_0 \rvert +\beta_1\rvert^{4}] < 1$
\\[1.5ex] VGARCH & \multicolumn{3}{c}{for any $r \in \N$: \quad $\beta_1 < 1$}
\\[1.5ex] NGARCH &  $\E[\lvert \alpha_1 (\epsilon_0 + \gamma_1)^2 + \beta_1 \rvert^r] < 1$ & $ \alpha_1 (1 + \gamma_1^2) + \beta_1 < 1$&  $ \E[\lvert \alpha_1 (\epsilon_0 + \gamma_1)^2 + \beta_1 \rvert^2] < 1$
\\[1.5ex] MGARCH & \multicolumn{3}{c}{for any $r \in \N$: \quad $\E[\exp(4r \lvert \omega/p + \alpha_1 \log(\epsilon_0^2)  \rvert^2)] < \infty$ and $\lvert \beta_1 \rvert <1$}
\\[1.5ex] EGARCH & \multicolumn{3}{c}{for any $r \in \N$: \quad $\E[\exp(4r  \lvert \omega/p + \alpha_1 (\lvert\epsilon_0\rvert - \E\lvert \epsilon_0 \rvert) + \gamma_1 \epsilon_0 \rvert^2)] < \infty$ and $\lvert \beta_1 \rvert <1$}
\\  [1ex]
\hline 
\end{tabular}
\end{center}
}
\end{table}
%
\begin{table}[h]
{\scriptsize
\caption{\label{tbl:cond_bivariate-conv-pq} \sf\footnotesize  Conditions $(P_{\max{(1, r/\delta)}})$ or $(L_r)$ respectively translated for different augmented GARCH($p$,$q$) models for the general $r$-th absolute centred sample moment, $r \in \N$.}
\vspace{-2ex}
\begin{center}
\begin{tabular}{l | l }
\hline
\\[-1.5ex]
augmented \\ GARCH ($p$,$q$) & $r \in \N$  \\
\hline\hline
\\[-1.5ex]
\\[-2ex] APGARCH &  $\sum_{j=1}^{max(p,q)} \E[\lvert \alpha_j \left(\lvert \epsilon_0\rvert - \gamma_j \epsilon_{t-j}\right)^{2 \delta} + \beta_j \rvert^r]^{1/r} <1$ 
\\[1.5ex] AGARCH & $\sum_{j=1}^{max(p,q)} \E[\lvert \left( \alpha_j \lvert \epsilon_0\rvert - \gamma_j \epsilon_{t-j}\right)^{2 } + \beta_j \rvert^r]^{1/r} <1$ 
\\[1.5ex] GJR-GARCH & $\sum_{j=1}^{max(p,q)} \E[\lvert \alpha_j^{\ast} \epsilon_0^2 + \beta_j + \gamma_j^{\ast} \max(0,-\epsilon_0^2) \rvert^r]^{1/r} <1$
\\[1.5ex] GARCH & $\sum_{j=1}^{\max(p,q)} \E[(\alpha_j \epsilon_0^2 + \beta_j)^r]^{1/r} < 1 $ 
\\[1.5ex] ARCH & $\sum_{j=1}^{\max(p,q)} \alpha_j \E[ \epsilon_0^{2r}]^{1/r} < 1 $ 
\\[1.5ex] TGARCH & $\sum_{j=1}^{\max(p,q)} \E[\lvert \alpha_j \lvert \epsilon_{t-j} \rvert - \alpha_j \gamma_j \epsilon_{t-j} + \beta_j \rvert^r]^{1/r} < 1$
\\[1.5ex] TSGARCH & $\sum_{j=1}^{\max(p,q)} \E[\lvert \alpha_j \lvert \epsilon_{t-j} \rvert + \beta_j \rvert^r]^{1/r} < 1$
\\[1.5ex] PGARCH & $\sum_{j=1}^{max(p,q)} \E[\lvert \alpha_j \lvert \epsilon_0 \rvert +\beta_j\rvert^{2r}]^{1/(2r)} < 1$ 
\\[1.5ex] VGARCH & $\sum_{j=1}^q \beta_j < 1$
\\[1.5ex] NGARCH & $\sum_{j=1}^{\max(p,q)} \E[\lvert \alpha_j (\epsilon_0 + \gamma_j)^2 + \beta_j \rvert^r]^{1/r} < 1$ 
\\[1.5ex] MGARCH & $\E[\exp(4r \sum_{i=1}^p  \lvert \omega/p + \alpha_i \log(\epsilon_0^2)  \rvert^2)] < \infty$ and $\sum_{j=1}^q \lvert \beta_j \rvert <1$
\\[1.5ex] EGARCH & $\E[\exp(4r \sum_{i=1}^p  \lvert \omega/p + \alpha_i (\lvert\epsilon_0\rvert - \E\lvert \epsilon_0 \rvert) + \gamma_i \epsilon_0 \rvert^2] < \infty$ and $\sum_{j=1}^q \lvert \beta_j \rvert <1$
\\  [1ex]
\hline
\end{tabular}
\end{center}
}
\end{table}

In Table~\ref{tbl:cond_bivariate-conv-11}, we consider in the first column the conditions for the general $r$-th absolute centred sample moment, $r \in \N$.
We also specifically look at the standard cases of the sample MAD ($r=1$) and the sample variance ($r=2$) as measure of dispersion estimators respectively, presented in the second and third column.

For the selected polynomial GARCH models, the requirement ${\sum_{i=1}^p \| g_i(\epsilon_0) \|_{\max(1,r/\delta)} < \infty}$ in condition $(P_{\max(1,r/\delta)})$ will always be fulfilled. Thus, we only need to analyse the condition\\
$\displaystyle {\sum_{j=1}^q \| c_j (\epsilon_0)  \|_{\max(1,r/\delta)} < 1}$.

Lastly, we present in Table~\ref{tbl:cond_bivariate-conv-pq} how the conditions $(P_{\max{(1, r/\delta)}})$ or $(L_r)$ respectively translate for those augmented GARCH($p$,$q$) processes - this is the generalization of Table~\ref{tbl:cond_bivariate-conv-11}.
As, in contrast to Table~\ref{tbl:cond_bivariate-conv-11}, we do not gain any insight by considering the choices of $r=1$ or $r=2$, we only present the general case, $r \in \N$.

When $p \neq q$, we need to consider coefficients $\alpha_j ,\beta_j, \gamma_j$ for $j=1,...,\max{(p,q)}$. In case they are not defined, we set them equal to 0. 

Note that, in Table~\ref{tbl:cond_bivariate-conv-11} (and also Table~\ref{tbl:cond_bivariate-conv-pq}), the restrictions on the parameter space, given by $(P_{\max(1,r/\delta)})$ or $(L_r)$ respectively, are the same as the conditions for univariate FCLTs of the process $X_t^r$ itself (see \cite{Berkes08}, \cite{Hormann08}). For $r=1$, they coincide with the conditions for e.g. $\beta$-mixing with exponential decay (see \cite{Carrasco02}).
  \end{appendix}

\end{document}